\documentclass[10pt]{amsart}
\usepackage[latin1]{inputenc}
\usepackage[all]{xy}
\usepackage{amssymb, amsmath, amscd, geometry}
\usepackage{graphicx}
\usepackage{setspace}

\numberwithin{equation}{section}
\input xy
\xyoption{all}
\usepackage{latexsym}
\usepackage[usenames]{color}
\usepackage{epic} %%%% Circunferencias grandes y curvas de bezier
\usepackage{mathrsfs}
\usepackage{euscript}
\usepackage{rotating}
\usepackage{comment}

\theoremstyle{plain}
\newtheorem{lemma}{Lemma}[section]
\newtheorem{theorem}[lemma]{Theorem}
\newtheorem{corollary}[lemma]{Corollary}

\newtheorem{prop}[lemma]{Proposition}

\theoremstyle{definition}
\newtheorem{definition}{Definition}[section]
\newtheorem{remark}{Remark}[section]

      \newcommand{\N}{{\mathbb N}}
      
      \newcommand{\C}{{\mathbb C}}

     \newcommand{\Ne}{{\mathcal N}}

\newcommand{\tr}{\operatorname{tr}}

\newcommand{\id}{\operatorname{id}}

\newcommand{\rank}{\operatorname{rank}}

\newcommand{\ba}{|}

\newcommand{\uno}{1\!\!1}

\allowdisplaybreaks

\begin{document}

\baselineskip=14pt

\title[CB-norm estimates for maps between noncommutative $L_p$-spaces]{CB-norm estimates for maps between noncommutative $L_p$-spaces and quantum channel theory}

\author{Marius Junge and Carlos Palazuelos}

\thanks{The first author is partially supported by NSF DMS-1201886 grant. The second author is partially supported by MINECO (grant MTM2011-26912), the european CHIST-ERA project CQC (funded partially by MINECO grant PRI-PIMCHI-2011-1071) and ``Ram\'on y Cajal'' program. Both authors are partially supported by ICMAT Severo Ochoa Grant SEV-2011-0087 (Spain).}

\begin{abstract}
In the first part of this work we show how certain techniques from quantum information theory can be used in order to obtain very sharp embeddings between noncommutative $L_p$-spaces. Then, we use these estimates to study the classical capacity with restricted assisted entanglement of the quantum erasure channel and the quantum depolarizing channel. In particular, we exactly compute the capacity of the first one and we show that certain nonmultiplicative results hold for the second one.
\end{abstract}

\maketitle

\section{Introduction}
Embedding results  for $L_p$-spaces have a very long tradition in Banach space theory, see e.g. the handbook \cite{JoLi}. In some sense the starting point are the probabilistic concepts of $p$-stable random variables going back at least as early as \cite{lev}. Noncommutative analogues of such embedding results have been established by imitating and modifying the commutative results 
\cite{Jun, JunPar I, JunPar II}. The novelty in this paper is to use what should be called ``classical ideas'' from the emerging new quantum information theory and significantly improve embedding results for (vector-valued) noncommutative $L_p$-spaces, and indicate some applications. On the other hand, operator algebra and functional analysis techniques  have been very successfully applied in quantum information theory. For example, operator space techniques have been applied to Bell inequalities (\cite{JuPa1}, \cite{JPPVW2}, \cite{PWPVJ}), tools from free probability have been used for the classical capacity of a quantum channel (\cite{BCN}, \cite{BeNe I}, \cite{BeNe II}), and noncommutative versions of Grothendieck theorem where used for efficient approximations for quantum values of quantum games (\cite{CJPP}, \cite{ReVi}). There are also some examples using techniques from quantum information to prove new mathematical results. For example Regev and Vidick used the embezzlement state for a simplified proof of the so called Grothendieck theorem for operator spaces (\cite{ReVi2}) and Ahlswede/Winter's application of the Goldon-Thompson inequality has found numerous application in compressed sensing (see \cite{Rau}).

In this paper we will use quantum teleportation, one of the most important quantum information protocols, to provide some very sharp embeddings between noncommutative $L_p$-spaces. Let us recall the definition of the \emph{discrete noncommutative vector valued $L_p$-spaces}, introduced by Pisier in \cite{Pisierbook2}. For a given natural number $n$ and $1\leq p\leq \infty$ we will denote by $S_p^n:=S_p(\ell_2^n)$ the \emph{Schatten $p$-class} of operators acting on the $n$-dimensional complex Hilbert space $\ell_2^n$, which can be obtained by interpolation: $S_p^n=\big[S_\infty^n, S_1^n\big]_{\frac{1}{p}}$, where $S_\infty^n$ denotes the space of (compact) operators acting on $\ell_2^n$ joint with the operator norm and the trace class $S_1^n$ can be seen as the dual space of $S_\infty^n$ with respect to the dual action $\langle A,B\rangle=tr(AB^t)$. In fact, such an interpolation identity can be used to endow the space $S_p^n$ with a natural \emph{operator space structure} (\cite{Pisierbook}, \cite{Pisierbook2}). Note that the diagonal of $S_p^n$ is exactly $\ell_p^n=\big[\ell_\infty^n, \ell_1^n\big]_{\frac{1}{p}}$, so one also obtains an operator space structure for these spaces. An \emph{operator space} $E$ is a complex Banach space together with a sequence of \emph{matrix norms} $\alpha_n$ on $M_n[E]=M_n\otimes E$ with $n\geq 1$, satisfying certain ``good properties''. Then, given a linear map $T:E\rightarrow F$ between operator spaces we say that $T$ is a complete contraction (resp. a complete isomorphism/complete isometry) if the maps $id_{M_n}\otimes T:M_n[E]\rightarrow M_n[E]$ are contractions (resp. isomorphisms/isometries) for every $n$. When working with operator spaces these are precisely the morphisms one has to use in order to preserve the new structure. Finally, given any operator space $E$, we will denote $S_\infty[E]=S_\infty\otimes_{min} E$, where $\min$ denotes the minimal tensor norm in the category of operator spaces. On the other hand, Effros and Ruan introduced the space $S_1[E]$ as the (operator) space $S_1\widehat \otimes E$, where $\widehat \otimes$ denotes the projective operator space tensor norm. Then, using complex interpolation Pisier defined the noncommutative vector valued (operator) space $S_p[E]=\big[S_\infty[E], S_1[E]\big]_{\frac{1}{p}}$ for any $1\leq p\leq \infty$ and he proved that this definition leads to obtain the expected properties of $S_p[E]$, analogous to the commutative setting (see \cite[Chapter 3]{Pisierbook2}). The first result of this work is the following.
\begin{theorem}\label{Theorem I: teleportation-general}
Let $1\leq p,q\leq \infty$. Let $n_1,\cdots ,n_k$ be a family of natural numbers and let $d$ be the least common multiplier of $n_1,\cdots ,n_k$. There exist a completely positive and completely isometric embedding$$\tilde{J}_{p,q}:S_p^{n_1}\oplus_p \cdots \oplus_p S_p^{n_k}\rightarrow S_q^d\big(\ell_p^{n_1^2+\cdots +n_k^2}\big)$$ and  a completely positive and completely contractive map $$\tilde{W}_{p,q}:S_q^d\big(\ell_p^{n_1^2+\cdots +n_k^2}\big) \rightarrow S_p^{n_1}\oplus_p \cdots \oplus_p S_p^{n_k}$$ such that $ \tilde{W}_{p,q}\circ \tilde{J}_{p,q}=id$.

Moreover, the result is also true in the vector valued setting. That is, for any operator space $E$, $S_p^{n_1}[E]\oplus_p \cdots_p \oplus S_p^{n_k}[E]$ is completely isometric to a completely complemented subspace of $S_q^d\big(\ell_p^{n_1^2+\cdots +n_k^2}[E]\big)$.
\end{theorem}
Finding suitable embeddings of vector valued $L_p$-spaces has a long tradition in Banch space theory, and can be used in noncommutative harmonic analysis, quantum probability theory and operator spaces (see for instance  \cite{JunPar I}, \cite{JunPar II}, \cite{JunPar III} and the references therein). In particular, the type of embeddings given in Theorem \ref{Theorem I: teleportation-general}  has  been used in order to study notions like type and cotype or $K$-convexity and $B$-convexity in the context of operator spaces. This is the case of the work \cite{JunPar I}, where the authors, motivated by the study of the previous notions, provided a complete isomorphism from the space $S_p^n$ onto a completely complemented subspace of $S_q^{n^m}(\ell_p^m)$ with $m\approx n^2$ (\cite[Theorem 2]{JunPar I}). Moreover, using type/cotype estimates they proved that the order  $m\approx n^2$ is optimal\footnote{Remarkably, this order is different from the well known optimal commutative order $m
 \approx n$.}. An immediate corollary of Theorem \ref{Theorem I: teleportation-general} is the following result, which significantly improves \cite[Theorem 2]{JunPar I}.
\begin{corollary}\label{Cor teleport 1 space}
Let $1\leq p,q\leq \infty$. There exists a complete isometry of $S_p^n$ onto a completely complemented subspace of $S_q^{n}(\ell_p^{n^2})$. Moreover, both the isometry and the projection are completely positive maps. The result also holds in the vector valued case.
\end{corollary}
Hence, while keeping the optimal order $n^2$ in the commutative part ($\ell_p$-space) Corollary \ref{Cor teleport 1 space} provides a very tight estimate for the dimension of the noncommutative part ($S_q$-space). Moreover, we have now a complete isometry rather than a complete isomorphism (where a universal constant $C$ appears in the relation of the norms).

Some preliminary calculations show that the techniques developed in this work could be used to define some new embeddings in more general contexts. However, since our main motivation in this work is the use of Theorem \ref{Theorem I: teleportation-general} to study the capacity of certain quantum channels, we postpone this analysis to a future publication.

Finally we will show the following result, which can be understood as a complement of Theorem \ref{Theorem I: teleportation-general}. The key point here is to use ideas from the superdense conding, another important protocol of quantum information.
\begin{theorem}\label{Theorem superdense embedding}
Let $1\leq p,q\leq \infty$. Then, there exist a completely positive and a completely isometric map $$H_{p,q}:\ell_p^{n^2}\rightarrow S_q^n(S_p^n)$$and a completely positive and completely contractive map $$Q_{p,q}:S_q^n(S_p^n)\rightarrow \ell_p^{n^2}$$ such that $Q_{p,q}\circ H_{p,q}=id$.

Moreover, if $E$ is any operator space, $\ell_p^{n^2}[E]$ is completely isometric to a completely complemented subspace of $S_q^n\big(S_p^n[E]\big)$.
\end{theorem}
A quantum channel is defined  as a completely positive and trace preserving map $\Ne:M_n\rightarrow M_m$. Following  \cite{JuPa2} we will denote a quantum channel by $\Ne:S_1^n\rightarrow S_1^m$, where we use $S_1^k$ to denote the trace class of operators acting on $\ell_2^k$. This notation emphasizes the idea that $\Ne$ must be, in particular, a norm one operator on these spaces. As it was shown in \cite{DJKR} and \cite{JuPa2}, one can understand some channel capacities as the derivative of certain completely bounded and completely $p$-summing norms. We refer to \cite[Section 5]{JuPa2} for a brief introduction about channel capacities from a mathematical point of view. In particular, if we denote by $C_{prod}^d(\Ne)$ the product state version of the classical capacity of the quantum channel $\Ne$ with assisted entanglement restricted to dimension $d$ per channel use, one can see that $C_{prod}^d(\Ne)$ can be written as the derivative (with respect to $p$) of the $\ell_p(S_p^d)$-summing norm of the adjoint map $\Ne^*:M_m\rightarrow M_n$ (see \cite[Theorem 1.1]{JuPa2} for details). Note that this family of capacities covers, in particular, the well studied classical capacity with non entanglement  ($d=1$) and the classical capacity with unlimited assisted entanglement ($d=n$). Unfortunately, in order to compute the corresponding capacity (rather than its product state version) one has to consider the regularization
\begin{align}\label{regularization}
C^d(\Ne):=\sup_k\frac{C_{prod}^{d^k}(\Ne^{\otimes _k})}{k}.
\end{align}
Since quantum information theory deals with the ways we can send and manipulate the information by using quantum resources, it is not surprising that the study of quantum channel capacities is one of the main topics in the theory and, so, it has captured the attention of many researchers in the area (see for instances \cite{SmithChannels} and the references therein). Let us consider here the \emph{quantum depolarizing channel} with parameter $\lambda\in [0,1]$, $\mathcal \mathcal D_\lambda:S_1^n\rightarrow S_1^n$, defined by $$\mathcal \mathcal D_\lambda(\rho)=\lambda\rho+ (1-\lambda)\frac{1}{n}tr(\rho)\uno_n \text{      }\text{    for every    }\text{      }  \rho\in S_1^n,$$and also the \emph{quantum erasure channel} with parameter $\lambda\in [0,1]$, $\mathcal E_\lambda:S_1^n\rightarrow S_1^n\oplus_1 \C$, defined by $$\mathcal E_\lambda(\rho)=\lambda\rho\oplus (1-\lambda)tr(\rho) \text{      }\text{    for every    }\text{      }  \rho\in S_1^n.$$Here $\uno_n$ denotes the identity element in $M_n$. The previous two channels are very important in quantum information because, despite its very simple form, they already provide some non trivial examples. In order to emphasize this idea, let us mention that computing the (non considered in this work) quantum capacity of the depolarizing channel (even in dimension $n=2$) is an open problem in the area (see \cite{Ouyang}, \cite{SmSm} for some recent progresses). On the other hand, the classical capacity of the $\mathcal \mathcal D_\lambda$ with no assisted entanglement ($C^1(\Ne)$) and with unlimited entanglement ($C^n(\Ne)$) are well understood (see \cite{King} and \cite{BSST2} respectively). The key point here is that both quantities, $C^1_{prod}$ and $C^n_{prod}$, are multiplicative when acting on the tensor product of depolarizing channels\footnote{In fact, it was shown in \cite{BSST2} that $C_{prod}^n$ is multiplicative on every channel so we always have $C^n=C_{prod}^n$.}, so the regularization (\ref{regularization}) is not required in this case. On the other hand, a very good property of these two channels is that they are covariant (see definition below) and that allows us to simplify the statement of \cite[Theorem 1.1]{JuPa2} so that one has to deal with the $d$-norm of the corresponding channel $$\big\|\Ne:S_1^n\rightarrow S_p^n\big\|_d:=\big\|id_{M_d}\otimes \Ne:M_d(S_1^n)\rightarrow M_d(S_p^n)\big\|,$$rather than with the $\ell_p(S_p^d)$-summing norm of the adjoint map $\Ne^*$. More precisely, for any covariant quantum channel $\Ne:S_1^n\rightarrow S_1^n$ we have
\begin{align*}
C_{prod}^d(\Ne)=\ln n+ \frac{d}{dp}[\|\Ne:S_1^n\rightarrow S_p^n\|_{d}]|_{p=1}
\end{align*}for every $1\leq d\leq n$ (\cite[Corollary 4.2]{JuPa2}). Then, we can use the estimate proved in Theorem \ref{Theorem I: teleportation-general} to obtain the following result.
\begin{theorem}\label{Main theorem Depolarizing}
Let $\mathcal \mathcal D_\lambda:S_1^n\rightarrow S_1^n$ and  $\mathcal E_\lambda:S_1^n\rightarrow S_1^n\oplus_1 \C$ be respectively the quantum depolarizing channel and the quantum erasure channel with parameter $\lambda\in [0,1]$ defined as before and let $d$ be a natural number such that $1\leq d\leq n$. Then,
\begin{align}\label{p-norm depol}
\big\|\mathcal \mathcal \mathcal D_\lambda:S_1^n\rightarrow S_p^n\big\|_d=\Big(\frac{1}{d}\big(\lambda d+\frac{1-\lambda}{n}\big)^p+\big(\frac{1-\lambda}{n}\big)^p\big(n-\frac{1}{d}\big)\Big)^\frac{1}{p},
\end{align}which implies
\begin{align*}
C_{prod}^d(\mathcal \mathcal D_\lambda)= \log_2  (nd)+ \big(\lambda+ \frac{1-\lambda}{nd}\big)\log_2  \big(\lambda+ \frac{1-\lambda}{nd}\big)+(nd-1)\big(\frac{1-\lambda}{nd}\big)\log_2  \big(\frac{1-\lambda}{nd}\big).
\end{align*}On the other hand,
\begin{align}\label{p-norm erasure}
\big\|\mathcal E_\lambda:S_1^n\rightarrow S_p^n\oplus_p \C\big \|_d=\Big(\lambda^pd^{p-1}+(1-\lambda)^p\Big)^\frac{1}{p},
\end{align}so that
\begin{align*}
C_{prod}^d(\mathcal E_\lambda)=\lambda \log_2  (nd).
\end{align*}
\end{theorem}
Both expressions $C_{prod}^d(\mathcal \mathcal D_\lambda)$ and $C_{prod}^d(\mathcal E_\lambda)$ extend the previously known expressions for the cases $d=1$ and $d=n$. This is very surprising in view of the fact that for the depolarizing channel the formula $C_{prod}^d(\mathcal \mathcal D_\lambda)$ is not multiplicative and, hence, $C^d(\mathcal \mathcal D_\lambda)$ does not coincide with $C_{prod}^d(\mathcal \mathcal D_\lambda)$. Indeed, we have the following corollary of the previous theorem.
\begin{corollary}\label{Corollary Depolarizing}
Let us fix $n=4$, $d=2$ and $\lambda\in (0,1)$. Then,
\begin{align*}
C_{prod}^{d^2}(\mathcal \mathcal D_\lambda\otimes \mathcal \mathcal D_\lambda)>2C_{prod}^d(\mathcal \mathcal D_\lambda).
\end{align*}Hence,
\begin{align*}
C^d(\mathcal \mathcal D_\lambda)>C_{prod}^d(\mathcal \mathcal D_\lambda).
\end{align*}
\end{corollary}
Interestingly, the quantity $C^d(\mathcal \mathcal D_\lambda)$ has been also studied in some other works by using different techniques (\cite{HsWi}, \cite{WiHs}) and its  exact value seems to be unknown. On the other hand, we will show that $C_{prod}^d$ is multiplicative on the quantum erasure channel $\mathcal E_\lambda$ and we will use this estimate to bound the value $C^d(\mathcal \mathcal D_\lambda)-C_{prod}^d(\mathcal \mathcal D_\lambda)$. More precisely, we will prove the following result.
\begin{theorem}\label{Main Theorem Erasure}
Let $\mathcal \mathcal D_\lambda:S_1^n\rightarrow S_1^n$ and  $\mathcal E_\lambda:S_1^n\rightarrow S_1^n\oplus_1 \C$ be respectively the quantum depolarizing channel and the quantum erasure channel with parameter $\lambda\in [0,1]$ and let $d$ be any natural number such that $1\leq d\leq n$. Then,
\begin{align}\label{sharp bounds dep}
\lambda \log_2 (nd)- H(\mu)\leq C_{prod}^d(\mathcal D_\lambda)\leq C^d(\mathcal D_\lambda)\leq \lambda \log_2 (nd).
\end{align}Here, $H(\mu)=-\mu\log_2(\mu)-(1-\mu)\log_2(1-\mu)$ is called the Shannon entropy of the probability distribution $(\mu, 1-\mu)$, where $\mu=\lambda+\frac{1-\lambda}{nd}$. In particular, $\lambda \log_2 (nd)- 1\leq C_{prod}^d(\mathcal D_\lambda)$. On the other hand,
\begin{align}\label{Eq multiplicativity erasure}
C^{d^k}_{prod}(\mathcal E_\lambda^{\otimes_k})= kC^d_{prod}(\mathcal E_\lambda)=k\lambda \log_2 (nd).
\end{align}Hence, $$C^d(\mathcal E_\lambda)=\lambda \log_2 (nd).$$
\end{theorem}
The paper is organized as follows. In Section \ref{Sec: Quantum teleportation revised} we will first introduce some basic notions about operator spaces and noncommutative $L_p$-spaces that we will use along the whole paper. Then, we will prove Theorem \ref{Theorem I: teleportation-general} and Theorem \ref{Theorem superdense embedding}. In Section \ref{Sec: Some results about covariant channels} we will introduce some basic notions about quantum channels and we will explain why computations are easier when we deal with covariant channels. Section \ref{Sec: Depolarizing channel} is devoted to analyzing the quantum depolarizing channel. There, we will prove those parts of Theorem  \ref{Main theorem Depolarizing} and Theorem \ref{Main Theorem Erasure} corresponding to this channel and we will also prove Corollary \ref{Corollary Depolarizing}. Finally, in Section \ref{Sec: Erasure channel} we will study the quantum erasure channel. In particular, we will show the second part of Theorem \ref{Main theorem Depolarizing} and Theorem \ref{Main Theorem Erasure}.
\section{Quantum teleportation revised: Some sharp embeddings between noncommutative $L_p$-spaces}\label{Sec: Quantum teleportation revised}
\subsection{Some basic notions about operator spaces and noncommutative $L_p$-spaces}\label{Sec: Basic notions}
In this section we introduce some basic concepts from operator space theory. We focus only on those aspects which are useful for this work and we direct the interested reader to the standard references \cite{EffrosRuan}, \cite{Pisierbook}. Given Hilbert spaces $\mathcal H$ and $\mathcal K$, we will denote by $B(\mathcal H, \mathcal K)$ the space of bounded operators from  $\mathcal H$ to $\mathcal K$ endowed with the standard operator norm. When $\mathcal H=\ell_2^n$ and $\mathcal K=\ell_2^m$ we will denote $M_{n,m}=B(\ell_2^n, \ell_2^m)$ and in the case where $n=m$ we will just write $M_n$.

An \emph{operator space} $E$ is a complex Banach space together with a sequence of \emph{matrix norms} $\|\cdot\|_k$ on $M_k[E]=M_k\otimes E$ satisfying the following conditions:
\begin{itemize}\label{OS properties}
\item $\|v \oplus w\|_{k+l} = \max\{\|v\|_k,\|w\|_l \}$ and
\item $\|\alpha w \beta\|_k \leq\|\alpha\| \, \|w\|_l \, \|\beta\|$
\end{itemize}
for all $v \in M_k[E]$, $w \in M_l[E]$, $\alpha \in M_{k,l}$, and $\beta \in M_{l,k}$. A simple, but important, example of an operator space is $M_n$ with its operator space structure given by the usual sequence of matrix norms $\|\cdot\|_k$ defined by the identification $M_k[M_n]=M_{kn}$.

To understand this theory, one needs to study the morphisms that preserve the operator space structure. In contrast to Banach space theory, where one needs to study the bounded maps between Banach spaces, in the theory of operator spaces we need to study the \emph{completely bounded maps}. Given operator spaces $E$ and $F$ and a linear map $T:E \to F$, let $T_k: M_k[E] \to M_k[W]$ denote the linear map defined by
\[
T_k(v)=(id_k \otimes T)(v)=(T(v_{ij}))_{i,j}.
\]
The map $T$ is said to be \emph{completely bounded} if
\[
\|T\|_{cb} = \sup_n \|T_k\|< \infty,
\]
and this quantity is then called the \emph{completely bounded} norm of $T$. We will say that $T$ is \emph{completely contractive} if $\|T\|_{cb}\leq 1$. Moreover, $T$ is said to be a \emph{complete isomorphism} (resp. \emph{complete isometry}) if each map $T_k$ is an isomorphism (resp.  an isometry).

As in Banach space theory, we can also consider the notion of duality. Given an operator space $E$, we define the \emph{dual operator space} $E^*$ by means of the acceptable matrix norms
\begin{align*}
M_k[E^*]=CB(E,M_k), \text{    } k\geq 1.
\end{align*}
If we denote by $S_1^n$ the space $M_n$ with the trace norm, the duality relation $S_1^n=M_n^*$ allows us to define a natural operator space structure on $S_1^n$. This operator space structure is not given by the linear map identifying matrices in $S_1^n$ with matrices in $M_n$, but the right duality action is the scalar pairing
\begin{align}\label{Scalar Pairing}
\langle B, C\rangle =\mathrm{tr}\, (BC^{tr}),
\end{align}
which yields completely isometric isomorphisms $M_n^* = S_1^n$ and $(S_1^n)^* = M_n$. It is not difficult to see that $\|T^*\|_{cb}=\|T\|_{cb}$ for every $T:E \to F$, where $T^*$ denotes the adjoint map of $T$.

There is an equivalent definition of operator spaces, as those closed subspaces of $B(\mathcal H)$. On the one hand, given a subspace $E\subset B(\mathcal H)$ it is clear that we have a family of matrix norms, by identifying $M_k[E]\subset M_k(B(\mathcal H))=B(\ell_2^k\otimes \mathcal H)$, which can be shown to be an \emph{acceptable} sequence of matrix norms. The converse statement is known as Ruan's Theorem and can be found in \cite[Theorem 2.3.5]{EffrosRuan}. This point of view is very suitable to define the minimal tensor product of operator spaces. Given two operator spaces $E \hookrightarrow B(\mathcal H_E)$ and $F \hookrightarrow B(\mathcal H_F)$, we have a natural algebraic embedding of $E \otimes F$ in $B(\mathcal H_E \otimes \mathcal H_F)$. The \emph{minimal operator space tensor product} $E \otimes_{min} F$ is the closure of $E \otimes F$ in $B(\mathcal H_E \otimes \mathcal H_F)$. In particular, for every operator space $E$, one has that $M_n[E]=M_n\otimes_{min} E!
 $ isometr
 ically. One can check that for a couple of linear maps $T_1:E_1\rightarrow F_1$ and $T_2:E_2\rightarrow F_2$ one has
\begin{align}\label{mmp min norm}
\|T_1\otimes T_2:E_1\otimes_{min}E_2\rightarrow F_1\otimes_{min}F_2\|_{cb}=\|T_1:E_1\rightarrow F_1\|_{cb}\|T_2:E_2\rightarrow F_2\|_{cb}
\end{align}and that this tensor norm is commutative and associative (see \cite[Chapter 2]{Pisierbook}). Moreover, if $E$ and $F$ are finite dimensional, one can also check that we have the following completely isometric identification.
\begin{align*}
E \otimes_{min} F = CB(E^*,F),
\end{align*}where here the correspondence is defined by $\big(\sum_{i=1}^nv_i\otimes w_i\big)(v^*)=\sum_{i=1}^n\langle v_i,v^*\rangle w_i$.

The dual tensor norm of the minimal one is the so called \emph{projective tensor norm} (see \cite[Chapter 4]{Pisierbook}), $E\hat{\otimes} F$, which is defined for a given element $t\in M_k(E\otimes F)$, as
\begin{align*}
\|t\|_{M_k(E\hat{\otimes} F)}=\big\{\|\alpha\|_{M_{n,lm}}\|x\|_{M_l(E)}\|y\|_{M_m(E)}\|\beta\|_{M_{lm,n}}\big\},
\end{align*}where the infimum runs over all possible representations $t_{r,s}=\sum_{i,p,j,q}\alpha_{r,ip}\big(x_{ij}\otimes y_{pq})\beta_{jq,s}$ with $1\leq r\leq k$, $1\leq s\leq k$. This norm is also commutative and associative and for every finite dimensional operator spaces $E$ and $F$ one has the complete isometric identifications $$(E\otimes_{min}F)^*=E^*\hat{\otimes} F^* \text{       }\text{   and     } \text{       }(E\hat{\otimes}F)^*=E^*\otimes_{min} F^* .$$In particular, if we denote $S_1^n\hat{\otimes} E:=S_1^n[E]$, one has the completely isometric identification $\big(M_n[E]\big)^*=S_1^N[E^*]$. One can also check that
\begin{align}\label{mmp projective norm}
\|T_1\otimes T_2:E_1\hat{\otimes}E_2\rightarrow F_1\hat{\otimes}F_2\|_{cb}=\|T_1:E_1\rightarrow F_1\|_{cb}\|T_2:E_2\rightarrow F_2\|_{cb}
\end{align}for all linear maps $T_1:E_1\rightarrow F_1$ and $T_2:E_2\rightarrow F_2$.

Finally, given two operator spaces $E_0$ and $E_1$ which are compatible interpolation spaces in the sense of \cite[Section 2]{Pisierbook3}, one can define a natural operator spaces structure on $E_\theta=(E_0,E_1)_\theta$ by defining the following family of acceptable norms
\begin{align*}
M_k[E_\theta]=\big(M_k[E_0],M_k[E_1]\big)_\theta, \text{       }  k\geq 1.
\end{align*}As we explained in the introduction, this allows us to define a natural operator space structure on $S_p=(S_\infty, S_1)_{\frac{1}{p}}$ (resp. $S_p^n=(M_n, S_1^n)_{\frac{1}{p}}$) and, moreover, on $S_p[E]=(S_\infty[E], S_1[E])_{\frac{1}{p}}$ (resp. $S_p^n[E]=(M_n[E], S_1^n[E])_{\frac{1}{p}}$) for every operator space $E$. Here, $S_\infty$ denotes the space of compact operators on $\ell_2$ with the operator norm. As a particular case, the previous interpolation formula allows us to talk about the $p$-direct sum of operator spaces $\ell_p^n(E_i):=E_1\oplus_p \cdots \oplus_p E_n$. If $1<p<\infty$ and $\theta=\frac{1}{p}$, then for any compatible couple of operator spaces  $(E_0,E_1)$ the previous definition yields to the completely isometric identification $S_{p_\theta}[E_\theta]=(S_{p_0}[E_0],S_{p_1}[E_1])_{\theta}$, where $\frac{1}{\theta}=\frac{1-\theta}{p_0}+\frac{\theta}{p_1}$ and $E_\theta=(E_0,E_1)_\theta$ (see \cite[Theorem 1.1]{Pisierbook2}). Moreover, it w!
 as shown
 in \cite{Pisierbook2} that this definition of noncommutative $L_p$-spaces leads to the expected properties analogous to the classical ones. A very useful result, analogous to the classical case, states that given two couples of operator spaces ($E_0, E_1$) and ($F_0, F_1$), one has that
\begin{align}\label{interpolation step I}
\|T\|_{CB(E_\theta, F_\theta)}\leq \|T\|_{CB(E_0, F_0)}^{1-\theta}\|T\|_{CB(E_1, F_1)}^{\theta}.
\end{align}
According to the previous definition of the operator spaces $S_p[E]$ ($1\leq p<\infty$), it can be seen (\cite[Lemma 1.7]{Pisierbook2}, \cite[Theorem 1.5]{Pisierbook2}) that
\begin{align*}
\|Y\|_{M_d[S_p]}=\sup_{A,B\in B_{S_{2p}^d}}\big\|(A \otimes \uno)Y(B\otimes \uno)\big\|_{S_p(\ell_2^d\otimes_2\ell_2)}.
\end{align*}
and
\begin{align*}
\|X\|_{S_p[E]}=\inf \big\{\|A\|_{S_{2p}}\|Z\|_{B(\ell_2)\otimes_{\min}E}\|B\|_{S_{2p}}\big\},
\end{align*}where the last infimum runs over all representations of the form $X=\big(A\otimes \uno \big)Z\big(B\otimes \uno \big)$. Here, $B_{S_{2p}^d}$ denotes the unit ball $S_{2p}^d$ and $\uno$ denotes the identity operator in $B(\ell_2)$. We will usually denote by $\uno_n$ the identity matrix in $M_n$ appearing in the corresponding formulae for $\|Y\|_{M_d[S_p^n]}$ and $\|X\|_{S_p[E]}$.

In the second part of this work, we will mainly deal with the case $E=S_q^d$ for some $1\leq q\leq \infty$. It can be seen that, given $1\leq p,q\leq \infty$ and defining $\frac{1}{r}=|\frac{1}{p}-\frac{1}{q}|$, we have:

\noindent If $p\leq q$,
\begin{align}\label{norm p< q}
\|X\|_{S_p^n[S_q^d]}=\inf\Big\{\|A\|_{S_{2r}^n}\|Y\|_{S_q^{nd}}\|B\|_{S_{2r}^n}\Big\},
\end{align}where the infimum runs over all representations $X=(A\otimes \uno_d)Y(B\otimes \uno_d)$ with $A,B\in M_n$ and $Y\in M_n\otimes M_d$.

\noindent If $p\geq q$,
\begin{align}\label{norm p>q}
\|X\|_{S_p^n[S_q^d]}=\sup\Big\{\big\|(A\otimes \uno_d)X(B\otimes \uno_d)\big\|_{S_{q}^{nd}}: A,B\in B_{S_{2r}^n}\Big\}.
\end{align}
As an interesting application of this expression for the norm in $S_p[S_q]$ in \cite[Theorem 1.5 and Lemma 1.7]{Pisierbook2} Pisier showed that for a given linear map between operator spaces $T:E\rightarrow F$ we can compute its completely bounded norm as
\begin{align}\label{CB norm with general $t$}
\|T\|_{cb}=\sup_ {d\in \N}\big\|id_d\otimes T:S_t^d[E]\rightarrow S_t^d[F]\big\|
\end{align}for every $1\leq t\leq \infty$. That is, we can replace $\infty$ with any $1\leq t\leq \infty$ in order to compute the cb-norm.
\begin{remark}\label{(p,1)}
It is known (\cite{Aud2}, \cite{War}) that if $T$ is completely positive we can compute $\|T:S_q\rightarrow S_p\|$ by restricting to positive elements $A\in S_q$. Moreover, in this case one can also consider positive elements $X\geq 0$ to compute the cb-norm of $T$ (\cite[Section 3]{DJKR}) $\|T\|_{cb}=\|id_{S_q}\otimes T:S_q[S_q]\rightarrow S_q[S_p]\|$. On the other hand, for a positive element $X$, one can consider $A=B> 0$ in the expressions (\ref{norm p< q}) and  (\ref{norm p>q}) for $\|X\|_{S_p^n[S_q^d]}$. According to this, if $X> 0$ and $q=1$, (\ref{norm p>q}) becomes
\begin{align*}
\|X\|_{S_p^n[S_1^d]}=\sup_{A>0}\frac{\|(A\otimes \uno_d)X(A\otimes \uno_d)\|_{S_1^{nd}}}{\|A\|_{2p'}^2}
%\\\nonumber =\sup_{A>0}\frac{tr \big(A^2 (id_n\otimes tr_d)(X)\big) }{\|A^2\|_{p'}}
=\|(id_n\otimes tr_d)(X)\|_p,
\end{align*}where $\frac{1}{p}+\frac{1}{p'}=1$. Here and in the rest of the work we use notation $tr_n:=tr_{M_n}$.
\end{remark}
\subsection{General quantum teleportation}
We will start this section by introducing a family of unitaries which will be crucial in the rest of the work. For all $k,l=1,\cdots ,n$ we define the following unitaries on $\ell_2^n$: $$u_k(e_j)=e^{\frac{2\pi i kj}{n}}e_j  \text{     }\text{   and   }  \text{     } v_l(e_j)=e_{l+j}, \text{     }\text{   for every    }  \text{     }  j=1,\cdots, n,$$where $l+j$ will be always understood mod $n$. In this sense, we will understand $u_{-k}=u_{n-k}$ and $v_{-l}=v_{n-l}$ for any $k,l=1,\cdots ,n$. We will denote $\psi_n:= \sum_{i=1}^ne_i\otimes e_i\in \ell_2^n\otimes \ell_2^n$ and $\overline{\psi}_n:= \frac{1}{\sqrt{n}}\sum_{i=1}^ne_i\otimes e_i\in \ell_2^n\otimes \ell_2^n$. The following properties of the previous unitaries will be very useful in our analysis.
\begin{prop}\label{basic properties}

\

\begin{enumerate}
\item[a)] Let $p$ be any natural number in $\{1,\cdots, n\}$. Then,
\begin{align*}
\sum_{k=1}^ne^{\frac{2\pi ikp}{n}}=n\delta_{p,n}.
\end{align*}
%
%\item[b)] Let $\overline{\rho}_n=\frac{1}{\sqrt{n}}\sum_{l=1}^ne_l$. Then, for every $j=1,\cdots, n$ we have $$e_j=\frac{1}{\sqrt{n}}\sum_{k=1}^ne^{-\frac{2\pi ijk}{n}}u_k(\overline{\rho}_n).$$
%
\item[b)] For every $j,k=1,\cdots ,n$, we have $$e_j\otimes e_k=\frac{1}{\sqrt{n}}\sum_{s=1}^ne^{-\frac{2\pi isj}{n}}(u_s\otimes v_{k-j})(\overline{\psi}_n).$$
\item[c)] Let us define $\eta_{k,l}=(u_k\otimes v_l)(\overline{\psi}_n)$ for every $j,k=1,\cdots ,n$. Then, $(\eta_{k,l})_{k,l=1}^n$ is an orthonormal basis of $\ell_2^{n^2}=\ell_2^n\otimes_2\ell_2^n$.
\item[d)] Let $h\in\ell_2^n$. Then,
\begin{align}\label{shift}
h\otimes \overline{\psi}_n=\frac{1}{n}\sum_{k,l=1}^n\eta_{k,l}\otimes T_{k,l}(h),
\end{align}where we denote $T_{k,l}=v_lu_{-k}$ for every $k,l$. In particular, for every operator $\rho:\ell_2^n\rightarrow \ell_2^n$ we have
\begin{align*}
\rho\otimes |\overline{\psi}_n\rangle\langle\overline{\psi}_n |=\frac{1}{n^2}\sum_{k,l=1}^n\sum_{k',l'=1}^n|\eta_{k,l}\rangle\langle\eta_{k',l'}|\otimes T_{k,l}\rho T_{k',l'}^*.
\end{align*}Here, given two elements $\alpha$, $\beta$ in a Hilbert space $H$, we denote $|\alpha\rangle\langle\beta |:H\rightarrow H$ the rank one operator defined by $|\alpha\rangle\langle\beta |(h)=\langle\beta |h\rangle \alpha$. In particular, $|\alpha\rangle\langle\alpha |$ is the rank-one projection on $\alpha$.
\end{enumerate}
\end{prop}
\begin{proof}
Part a) is trivial.
%In order to prove a) we first note that $\sum_{k=1}^ne^{\frac{2\pi ikn}{n}}=\sum_{k=1}^ne^{2\pi ik}=n$. On the other hand, if $p\neq n$ we have
%\begin{align*}
%\sum_{k=1}^ne^{\frac{2\pi ikp}{n}}=\sum_{k=1}^n\big(e^{\frac{2\pi ip}{n}}\big)^k=\sum_{k=0}^{n-1}\big(e^{\frac{2\pi ip}{n}}\big)^k=\frac{\big(e^{\frac{2\pi ip}{n}}\big)^{n}-1}{e^{\frac{2\pi ip}{n}}-1}=0.
%\end{align*}
For the part b), we have
\begin{align*}
\frac{1}{\sqrt{n}}\sum_{s=1}^ne^{-\frac{2\pi isj}{n}}(u_s\otimes v_{k-j})(\overline{\psi}_n)&=\frac{1}{n}\sum_{s,l=1}^ne^{-\frac{2\pi isj}{n}}e^{\frac{2\pi isl}{n}}e_l\otimes e_{l+k-j}\\&
=\frac{1}{n}\sum_{l=1}^nn\delta_{l,j}e_l\otimes e_{l+k-j}=e_j\otimes e_k.
\end{align*}
In order to show part c) we first note that the fact that $u_k$ and $v_l$ are unitaries on $\ell_2^n$ guarantees that $(u_k\otimes v_l)$ is a unitary on $\ell_2^{n^2}$ for every $k,l$. Hence, since $\|\overline{\psi}_n\|=1$ we conclude that $\|\eta_{k,l}\|=1$ for every $j,k$. On the other hand, it is very easy to see that these vectors are orthogonal. Indeed, we have that
\begin{align*}
\langle \eta_{k',l'}, \eta_{k,l}\rangle&%=\frac{1}{n}\sum_{s,s'=1}^n\langle u_{k'}\otimes v_{l'}(e_{s'}\otimes e_{s'}), u_k\otimes v_l(e_s\otimes e_s)\rangle\\&
=\frac{1}{n}\sum_{s,s'=1}^n e^{-\frac{2\pi ik's'}{n}}e^{\frac{2\pi iks}{n}}\langle e_{s'}\otimes e_{s'+l'} ,e_{s}\otimes e_{s+l}\rangle%
%\\&=\frac{1}{n}\sum_{s,s'=1}^n e^{\frac{2\pi i(ks-k's')}{n}}\langle e_{s'}, e_s\rangle \langle e_{s'+l'} ,e_{s+l}\rangle\\&
%=\frac{1}{n}\sum_{s=1}^n e^{\frac{2\pi is(k-k')}{n}}\langle e_{s+l'} ,e_{s+l}\rangle\\&=\delta_{l,l'}\frac{1}{n}\sum_{s=1}^n e^{\frac{2\pi is(k-k')}{n}}
=\delta_{l,l'}\delta_{k,k'}.
\end{align*}
Finally, in order to show part d), let us consider $h=\sum_{j=1}^nh_je_j$. According to part c) above we have
\begin{align*}
h\otimes \overline{\psi}_n&=\frac{1}{\sqrt{n}}\sum_{j,l=1}^nh_je_j\otimes e_l\otimes e_l\\&=
\frac{1}{n}\sum_{j,l=1}^nh_j\big(\sum_{k=1}^ne^{-\frac{2\pi ikj}{n}}(u_k\otimes v_{l-j})(\overline{\psi}_n)\big)\otimes e_l\\&
=\frac{1}{n}\sum_{j,l,k=1}^nh_je^{-\frac{2\pi ikj}{n}}\eta_{k,l}\otimes e_{l+j}.
\end{align*}
On the other hand, note that $T_{k,l}(h)=\sum_{j=1}^nh_je^{-\frac{2\pi ikj}{n}}e_{l+j}$. Therefore,
\begin{align*}
h\otimes \overline{\psi}_n=\frac{1}{n}\sum_{k,l=1}^n\eta_{k,l}\otimes v_lu_{-k}(h).
\end{align*}
The second part of the statement can be obtained straightforwardly from the first one just looking at rank one operators $\rho=|h\rangle\langle k |$.
\end{proof}
\begin{corollary}\label{orth-complementation}
The linear map
\begin{align*}
i:\ell_p^{n^2}\rightarrow S_p^n[S_p^n] , \text{     }\text{   defined by    }\text{      } e_{k,l}\mapsto |\eta_{k,l}\rangle\langle\eta_{k,l}|, \text{   }k,l=1,\cdots, n,
\end{align*}is completely positive and a complete isometry and the linear map
\begin{align*}
P:S_p^n[S_p^n]\rightarrow \ell_p^{n^2} \text{     }\text{   defined by    }\text{      } P(A)=\sum_{k,l=1}^n \langle\eta_{k,l}|A|\eta_{k,l}\rangle e_{k,l}, \text{    } \text{    } A\in S_p^n[S_p^n],
\end{align*}is completely positive, and it is a completely contractive projection onto the image of $i$.

Moreover, for every operator space $E$ the map $i\otimes id_E$ defines a complete isometry of $\ell_p^{n^2}[E]$ onto a subspace of $S_p^{n^2}[E]$ which is completely complemented via $P\otimes id_E$.
\end{corollary}
\begin{proof}
The proof is immediate from part c) of Proposition \ref{basic properties} (see for instance \cite[Corollary 1.3]{Pisierbook2}).
\end{proof}
For the following lemma we note that $|\psi_n\rangle\langle\psi_n|$ can be seen as an element of $M_n\otimes M_n$ by writing $|\psi_n\rangle\langle\psi_n|= \sum_{i,j=1}^ne_{i,j}\otimes e_{i,j}$\footnote{Note that here we are shifting the spaces: $|\psi_n\rangle\langle\psi_n|= \sum_{i,j=1}^n(e_i\otimes e_i)\otimes (e_j\otimes e_j)=\sum_{i,j=1}^n(e_i\otimes e_j)\otimes (e_i\otimes e_j)$.}. Moreover, we note that the corresponding map $S_1^n\rightarrow M_n$ is the identity map. Hence, $|\psi_n\rangle\langle\psi_n|$ is an element in the unit ball of $S_1^n\otimes_{min}M_n$.
\begin{lemma}\label{lemma: assist-ent estimate}
Let us define the linear map $\iota:M_n\rightarrow M_n\otimes M_n\otimes M_n$ as $$\rho\rightarrow \rho\otimes |\psi_n\rangle\langle\psi_n|.$$ Then, for every operator space $E$, $\iota$ verifies that
\begin{align*}
\big\|\iota\otimes id_E:S_1^n[E]\rightarrow S_1^n(S_1^n)[E]\otimes_{min}M_n\big\|_{cb}\leq 1.
\end{align*}
\end{lemma}
\begin{proof}
We must show that
\begin{align*}
\Big\|\iota\otimes id_E\otimes id_k:S_1^n[E]\otimes_{min}M_k\rightarrow S_1^n(S_1^n)[E]\otimes_{min}M_n\otimes_{min}M_k\Big\|\leq 1
\end{align*}for every $k$. To this end, let us consider an element $x$ in the unit ball of $S_1^n[E]\otimes_{min}M_k$. Now, it follows from the definition of $\iota$ that $$(\iota\otimes id_E\otimes id_k)(x)=x\otimes |\psi_n\rangle\langle\psi_n|.$$On the other hand, since $|\psi_n\rangle\langle\psi_n|$ is in the unit ball of $S_1^n\otimes_{min}M_n$, according to (\ref{mmp min norm}), we have that $x\otimes |\psi_n\rangle\langle\psi_n|$ is in the unit ball of $S_1^n(S_1^n)[E]\otimes_{min} M_n\otimes_{min} M_k$. Here, we have used that $S_1^n[E]\hat{\otimes} S_1^n=S_1^n(S_1^n)[E]$. This concludes the proof.
\end{proof}
\begin{prop}\label{prop:p-embedding}
Let us define the linear map $J:M_n\rightarrow \ell_\infty^{n^2}\otimes M_n$ by
\begin{align*}
J(\rho)=\frac{1}{n}\sum_{k,l=1}^ne_{k,l}\otimes T_{k,l}\rho T_{k,l}^*.
\end{align*}Then, $J$ is completely positive verifying, for every operator space $E$,
\begin{align}\label{p,q-embedding}
\big\|J\otimes id_E:S_p^n[E] \rightarrow S_q^n(\ell_p^{n^2}[E])\big\|_{cb}\leq \frac{1}{n^{1-\frac{1}{p}-\frac{1}{q}}}
\end{align}for every $1\leq p \leq q\leq \infty$.
\end{prop}
\begin{proof}
The fact that $J$ is linear and completely positive is very easy. On the other hand, since it is well known that
\begin{align*}
\big \|id_n\otimes id_X:M_n[X]\rightarrow S_q^n[X]\big \|_{cb}=n^\frac{1}{q}
\end{align*}for every operator space $X$, it suffices to show that
\begin{align}\label{p-embedding}
\big\|J\otimes id_E:S_p^n[E] \rightarrow M_n(\ell_p^{n^2}[E])\big\|_{cb}\leq \frac{1}{n^{1-\frac{1}{p}}}
\end{align}for every $1\leq p \leq \infty$.

In order to prove the previous estimate for the case $p=1$,
\begin{align}\label{equation embeding p=1}
\big\|J\otimes id_E:S_1^n [E]\rightarrow M_n(\ell_1^{n^2}[E])\big\|_{cb}\leq 1,
\end{align}we invoke part b) in Proposition \ref{basic properties} to understand the map $J$ as
\begin{align*}
J=(P\otimes id_n)\circ \iota:S_1^n\rightarrow S_1^n(S_1^n)\otimes_{min} M_n\rightarrow \ell_1^{n^2}\otimes_{min} M_n.
\end{align*}Here, the map $P$ was defined in Corollary \ref{orth-complementation} and the map $\iota$ was defined in Lemma \ref{lemma: assist-ent estimate}. Indeed, this identification can be checked by basic calculations
\begin{align*}
\big((P\otimes id_n)\circ \iota\big)(\rho)&=(P\otimes id_n)\big(\rho\otimes |\psi_n\rangle\langle\psi_n|\big)\\&=(P\otimes id_n)\Big(\frac{1}{n}\sum_{k,l=1}^n\sum_{k',l'=1}^n|\eta_{k,l}\rangle\langle\eta_{k',l'}|\otimes T_{k,l}\rho T_{k,l}^*\Big)\\&
=\frac{1}{n}\sum_{k,l=1}^n|\eta_{k,l}\rangle\langle\eta_{k,l}|\otimes T_{k,l}\rho T_{k,l}^*.
\end{align*}
Hence, the estimate (\ref{equation embeding p=1}) follows from Corollary \ref{orth-complementation} and Lemma \ref{lemma: assist-ent estimate}.

In order to show the case $p=\infty$, we just note that $J$ is a completely positive map between C$^*$-algebras. Then, it is well known (see for instance \cite[Corollary 2.9]{Paulsen}) that
\begin{align*}
\big\|J:M_n \rightarrow M_n(\ell_\infty^{n^2})\big\|_{cb}=\big\|J(\uno_n)\big\|_{M_n(\ell_\infty^{n^2})}=\frac{1}{n}.
\end{align*}Then, (\ref{mmp min norm}) immediately implies that
\begin{align}\label{equation embeding p=infty}
\big\|J\otimes id_E:M_n[E] \rightarrow M_n(\ell_\infty^{n^2}[E])\big\|_{cb}=\big\|J(\uno_n)\big\|_{M_n(\ell_\infty^{n^2})}=\frac{1}{n}.
\end{align}since $M_n[E]=M_n\otimes_{min} E$ and $M_n(\ell_\infty^{n^2}[E])=M_n(\ell_\infty^{n^2})\otimes_{min}E$.

Finally, the case $1< p<\infty$ follows from (\ref{equation embeding p=1}), (\ref{equation embeding p=infty}) and interpolation (\ref{interpolation step I}).
\begin{align*}
\big\|J\otimes id_E:S_p^n [E]\rightarrow M_n(\ell_p^{n^2}[E])\big\|_{cb}\leq \Big(\frac{1}{n}\Big)^{1-\frac{1}{p}},
\end{align*}where $S_p^n [E]=(M_n [E],S_1^n [E])_{\frac{1}{p}}$ and $M_n(\ell_p^{n^2}[E])=\big(M_n(\ell_\infty^{n^2}[E]),M_n(\ell_1^{n^2}[E])\big)_{\frac{1}{p}}$.
\end{proof}

\begin{prop}\label{prop:p,q-projection}
Let $W:M_n\otimes \ell_\infty^{n^2}\rightarrow M_n$ be the linear map defined by
\begin{align*}
W\Big(\sum_{k,l=1}^nA_{k,l}\otimes e_{k,l}\Big)=\frac{1}{n}\sum_{k,l=1}^n T_{k,l}^*A_{k,l}T_{k,l}
\end{align*}for every $(A_{k,l})_{k,l=1}^n\subset M_n$. Then, $W$ is completely positive and it verifies, for every operator space $E$,
\begin{align*}
\big\|W\otimes id_E: S_q^n(\ell_p^{n^2}[E])\rightarrow S_p^n[E]\big\|_{cb}\leq n^{1-\frac{1}{p}-\frac{1}{q}},
\end{align*}for every $1\leq p\leq q \leq\infty$.
\end{prop}
\begin{proof}
The fact that $W$ is a linear map is obvious. Moreover, $W$ is defined as a sum of completely positive maps $A\mapsto T_{k,l}^*A_{k,l}T_{k,l}$, so it is completely positive. On the other hand, since it is well known that
\begin{align*}
\big \|id_n\otimes id_X:S_q^n[X]\rightarrow S_p^n[X]\big \|_{cb}=n^{\frac{1}{p}-\frac{1}{q}}
\end{align*}for every operator space $X$, it suffices to show that
\begin{align}\label{prop:p-projection}
\big\|W\otimes id_E:S_p^n (\ell_p^{n^2}[E])\rightarrow S_p^n[E]\big\|_{cb}\leq n^{1-\frac{2}{p}}.
\end{align}for every $1\leq p \leq \infty$.

Let us first consider the case $p=1$. The fact that $nW$ is completely positive and trace preserving immediately implies that $nW$ is completely contractive from $S_1^n(\ell_1^{n^2})$ to $S_1^n$. Thus, we have
\begin{align}\label{prop:p-projection p=1}
\big\|W\otimes id_E: S_1^n(\ell_1^{n^2}[E])\rightarrow S_1^n[E]\big\|_{cb}\leq \frac{1}{n},
\end{align}since $S_1^n(\ell_1^{n^2}[E])=S_1^n(\ell_1^{n^2})\hat{\otimes} E$ and $S_1^n[E]=S_1^n\hat{\otimes} E$ (\ref{mmp min norm}).
If we consider $p=\infty$, we have a completely positive map between the $C^*$-algebras $\ell_\infty^{n^2}(M_n)$ and $M_n$. As we have said previously, the completely bounded norm is then attained in the unit. Again, we easily deduce from here that
\begin{align}\label{prop:p-projection p=infty}
\big\|W\otimes id_E:M_n(\ell_\infty^{n^2}[E])\rightarrow M_n[E]\big\|_{cb}=\Big\|\frac{1}{n}\sum_{k,l=1}^n\uno_n\Big\|_{M_n}=n.
\end{align}Equations (\ref{prop:p-projection p=1}) and (\ref{prop:p-projection p=infty}) allow us to obtain the estimate in (\ref{prop:p-projection}) for a general case $1\leq p \leq \infty$ by interpolation (\ref{interpolation step I}). Indeed, we have
\begin{align*}
\big\|W\otimes id_E:S_p^n (\ell_p^{n^2}[E])\rightarrow S_p^n[E]\big\|_{cb}\leq \Big(\frac{1}{n}\Big)^\frac{1}{p}n^{1-\frac{1}{p}}=n^{1-\frac{2}{p}},
\end{align*}where we have used that $S_p^n (\ell_p^{n^2}[E])=\big(M_n(\ell_\infty^{n^2}[E]), S_1^n(\ell_1^{n^2}[E])\big)_{\frac{1}{p}}$.
\end{proof}
Instead of proving Theorem \ref{Theorem I: teleportation-general} directly we will first show how to obtain Corollary \ref{Cor teleport 1 space}. Then, we will explain how to adapt such a proof to obtain Theorem \ref{Theorem I: teleportation-general}.
\begin{proof}[Proof of Corollary \ref{Cor teleport 1 space}]
It suffices to show the case $1\leq p\leq q\leq \infty$, since the other case can be obtained by duality.

Let us define the linear maps
\begin{align*}
J_{p,q}:=n^{1-\frac{1}{p}-\frac{1}{q}}J:M_n \rightarrow \ell_\infty^{n^2}\otimes M_n,
\end{align*}where $J$ was defined in Proposition \ref{prop:p-embedding}, and
\begin{align*}
W_{p,q}:=n^{\frac{1}{p}+\frac{1}{q}-1}W:\ell_\infty^{n^2}\otimes M_n\rightarrow M_n,
\end{align*}where $W$ was defined in Proposition \ref{prop:p-projection}. According to the previous propositions both maps are completely positive and they verify the estimates
\begin{align*}
\big\|J_{p,q}\otimes id_E:S_p^n [E]\rightarrow S_q^n(\ell_p^{n^2}[E])\big\|_{cb}\leq 1, \text{    } \text{  and   }  \text{    } \big\|W_{p,q}\otimes id_E:S_q^n(\ell_p^{n^2}[E]) \rightarrow S_p^n[E]\big\|_{cb}\leq 1
\end{align*}for every operator space $E$.
Therefore, it suffices to show the algebraic identification $W_{p,q} \circ J_{p,q}=\uno_n$.
This is very easy just noting that for every $\rho\in M_n$ we have that
\begin{align*}
W_{p,q} \big(J_{p,q}(\rho)\big)= W \big(J(\rho)\big)=W \Big(\frac{1}{n}\sum_{k,l=1}^ne_{k,l}\otimes T_{k,l}\rho T_{k,l}^*\Big)=\frac{1}{n^2}\sum_{k,l=1}^n\rho =\rho.
\end{align*}
\end{proof}
Quantum teleportation is a communication protocol between two people, Alice and Bob, where say Alice can transmit a qubit (basic unit in quantum information theory) to Bob, by just sending two classical bits of information if they are allowed to share a maximally entangled state during the protocol. From a mathematical point of view, this means that there exist a channel (completely positive and trace preserving map) $\mathcal E:S_1^2\otimes S_1^2\rightarrow \ell_1^4$ (Alice's encoder from quantum to classical information) and another channel  $\mathcal D:\ell_1^4\otimes S_1^2\rightarrow S_1^2$ (Bob's decoder from classical to quantum information) so that the following diagram commutes:
$$\xymatrix@R=0.75cm@C=1cm {{\ell_1^4\otimes S_1^2}\ar[r]^{id} &
{\ell_1^4\otimes S_1^2}\ar[dd]^{\mathcal D} \\{(S_1^2\otimes S_1^2)\otimes S_1^2}\ar[u]^{\mathcal E\otimes id_{S_1^2}}
 & {}\\{S_1^2}\ar[u]^{i}\ar[r]^{id} & {S_1^2}},$$where here the map $i:S_1^2\rightarrow S_1^2\otimes S_1^2$ is defined by $i(\rho)=\rho\otimes |\overline{\psi}_2\rangle\langle\overline{\psi}_2 |$. A careful study of the channels $\mathcal E$, $\mathcal D$ in the teleportation protocol (see for instance \cite[Section 1.3.7]{NC}) should help the reader to identify the maps used in the proof of Corollary \ref{Cor teleport 1 space} for the particular case $n=2$.

The proof of Theorem \ref{Theorem I: teleportation-general} is a generalization of the previous one. However, in this case we need to be more careful since we have to use the same state $|\psi_d\rangle\langle\psi_d|\in M_d\otimes M_d$ to define different maps. Let us start by noting that the element $\psi_d$ can be seen as a tensor product element. Indeed,
\begin{align*}
\psi_d= \sum_{i=1}^n\sum_{j=1}^{n_1}(e_i\otimes e_j)\otimes (e_i\otimes e_j)= \sum_{i=1}^n(e_i\otimes e_i)\otimes \sum_{j=1}^{n_1}(e_j\otimes e_j).
\end{align*}Therefore,
\begin{align}\label{decom max ent state}
|\psi_d\rangle\langle\psi_d|&=\sum_{i,i'=1}^n\sum_{j,j'=1}^{n_1}|(e_i\otimes e_j)\otimes (e_i\otimes e_j)\rangle\langle (e_{i'}\otimes e_{j'})\otimes (e_{i'}\otimes e_{j'})|
\\\nonumber &=\sum_{i,i'=1}^n|e_i\rangle\langle e_{i'}|\otimes |e_i\rangle\langle e_{i'}|\otimes \sum_{j,j'=1}^{n_1}|e_j\rangle\langle e_{j'}|\otimes |e_j\rangle\langle e_{j'}|\\\nonumber &=|\psi_n\rangle\langle\psi_n|\otimes |\psi_{n_1}\rangle\langle\psi_{n_1}|.
\end{align}Similarly, we have that $|\psi_d\rangle\langle\psi_d|=|\psi_m\rangle\langle\psi_m|\otimes |\psi_{m_1}\rangle\langle\psi_{m_1}|$.

We will also need a ``more sophisticated'' interpolation result here, which allows us to interpolate not just the spaces, but also the operators.  We will use the following result, which can be found in \cite{CwJa}.
\begin{theorem}\label{Thm general interpolation}
Let $\bar{S}$ denote the close strip $\{z: 0\leq Re(z)\leq 1\}$ in the complex plane and $A(\bar{S})$ the algebra of bounded  continuous functions on $\bar{S}$ that are analytic on the open strip $S$. Let $(E_0,E_1)$ and $(F_0,F_1)$ be two compatible couples of Banach spaces and $\{T_z\}_{z\in \bar{S}}$ be a family of operators on $E_0\cap E_1$ into $F_0+F_1$ such that for every $a\in E_0\cap E_1$ and $b^*\in (F_0+F_1)^*$, $\langle b^*,T_z(a)\rangle \in A(\bar{S})$, there exist constants $M_0$, $M_1$ so that $\sup_{j+it}\|T_z:E_j\rightarrow F_j\|\leq M_j$ for $j=0,1$, and for every $a\in E_0\cap E_1$ we have that $\{T_{it}(a)\}_t$ lies in a separable subspace of $F_0$. Then, $$\|T_{\theta}:(E_0,E_1)_\theta\rightarrow (F_0,F_1)_\theta\|\leq M_0^{1-\theta}M_1^\theta.$$
\end{theorem}
To simplify notation, we will show the proof of the main theorem for the case of two spaces and in the scalar case ($E=\C$), $S_p^n\oplus_p S_p^m$. The reader will see that exactly the same proof applies in the general case.
\begin{proof}[Proof of Theorem \ref{Theorem I: teleportation-general}]
Again, it suffices to show the result for the case $1\leq p\leq q\leq \infty$, since the general case can be then obtained by duality.
In order to prove the first part of the theorem, let $d$ be the least common multiplier of $m$ and $n$ so that $d=nn_1=mm_1$ for certain natural numbers $n_1$ and $m_1$. Let us denote by $P^k$, $J^k$ and $W^k$ the linear maps introduced in Corollary \ref{orth-complementation}, Proposition \ref{prop:p-embedding} and Proposition \ref{prop:p,q-projection} respectively, when they are defined in dimension $k$ equal $n$ or $m$.

Motivated by (\ref{decom max ent state}), we consider the projection $$\tilde{P}^n:M_n(M_d)\rightarrow \ell_\infty^{n^2}$$ defined as
\begin{align*}
\tilde{P}^n(\rho)=P^n\big((id_n\otimes tr_{n_1})(\rho)\big) \text{     } \text{     } \text{  for every    } \text{     } \text{     }  \rho\in M_n(M_d)=M_n(M_n\otimes M_{n_1}).
\end{align*}We define $\tilde{P}^m:M_m(M_d)\rightarrow \ell_\infty^{m^2}$ analogously. Moreover, for every $1\leq p\leq q\leq \infty$ we consider the linear map
\begin{align*}
\tilde{J}_{p,q}:M_n\oplus M_m\rightarrow \ell_\infty^{n^2}(M_d)\oplus \ell_\infty^{m^2}(M_d)=(\ell_\infty^n\oplus \ell_\infty^m)(M_d).
\end{align*}defined by
\begin{align*}
\tilde{J}_{p,q}(\rho_1\oplus \rho_2)= d^{-\frac{1}{q}}\Big[(n^{\frac{1}{p'}}\tilde{P}^n\otimes id_d)\big(\rho_1\otimes |\psi_d\rangle\langle\psi_d|\big)\oplus (m^{\frac{1}{p'}}\tilde{P}^m\otimes id_d)\big(\rho_2\otimes |\psi_d\rangle\langle\psi_d|\big)\Big].
\end{align*}According to (\ref{decom max ent state}) we have
\begin{align*}
\tilde{J}_{p,q}(\rho_1\oplus \rho_2)&=d^{-\frac{1}{q}}\Big[\frac{1}{n^\frac{1}{p}}\sum_{j=1}^{n_1}\sum_{k,l=1}^ne_{k,l}\otimes T_{k,l}\rho_1 T_{k,l}^*\otimes |e_j\rangle\langle e_{j}|\\&\oplus \frac{1}{m^\frac{1}{p}}\sum_{i=1}^{m_1}\sum_{k',l'=1}^me_{k',l'}\otimes T_{k',l'}\rho_2 T_{k',l'}^*\otimes |e_i\rangle\langle e_i|\Big].
\end{align*}
$\tilde{J}_{p,q}$  is a direct sum of two completely positive maps. Thus, it is completely positive. We claim that
\begin{align}\label{p,q-embedding direct sum}
\big\|\tilde{J}_{p,q}:S_p^n\oplus_p S_p^m\rightarrow S_q^d(\ell_p^{n^2}\oplus_p \ell_p^{m^2})\big\|_{cb}\leq 1.
\end{align}
As we explained before, it suffices to show that
\begin{align}\label{p-embedding direct sum}
\big\|d^\frac{1}{q}\tilde{J}_{p,q}:S_p^n\oplus_p S_p^m\rightarrow M_d\otimes_{min}(\ell_p^{n^2}\oplus_p \ell_p^{m^2})\big\|_{cb}\leq 1.
\end{align}Since $d^\frac{1}{q}\tilde{J}_{p,q}$ does not depend on $q$, let us just denote $\tilde{J}_{p}$ this map.
Indeed, for the case $p=1$ we invoke the same argument as in the proof of Proposition \ref{prop:p-embedding} to state that the map
\begin{align*}
\tilde{\iota}:S_1^n\oplus_1 S_1^m\rightarrow (S_1^n\oplus_1 S_1^m)(S_1^d)\otimes_{min} M_d=(S_1^{nd}\oplus_1 S_1^{md})\otimes_{min} M_d,
\end{align*}defined by
\begin{align*}
\tilde{\iota}(\rho_1\oplus \rho_2)=(\rho_1\oplus \rho_2)\otimes |\psi_d\rangle\langle\psi_d|,
\end{align*}
is completely contractive. On the other hand, since $\tilde{P}^n:S_1^{nd}\rightarrow \ell_1^{n^2}$ and $\tilde{P}^m:S_1^{md}\rightarrow \ell_1^{m^2}$ are completely contractive maps, we conclude (see for instance \cite[Chapter 2]{Pisierbook2}) that
\begin{align*}
(\tilde{P}^n\oplus \tilde{P}^m)\otimes id_d :(S_1^{nd}\oplus_1 S_1^{md})\otimes_{min} M_d\rightarrow (\ell_1^{n^2}\oplus_1 \ell_1^{m^2})\otimes_{min} M_d
\end{align*}is a complete contraction. Since $\tilde{J}_1=\big((\tilde{P}^n\oplus \tilde{P}^m)\otimes id_d\big)\circ \tilde{\iota}$, we obtain that
\begin{align}\label{General proof int p=1}
\big\|\tilde{J}_1:S_1^n\oplus_1 S_1^m\rightarrow M_d\otimes_{min}(\ell_1^{n^2}\oplus_1\ell_1^{m^2})\big\|_{cb}\leq 1.
\end{align}
For the case $p=\infty$ we can proceed as in some previous proofs (just by evaluating the norm of $\tilde{J}_\infty(\uno_n\oplus \uno_m)$) or we can realized that, since  $\tilde{J}_\infty:M_n\oplus_\infty M_m\rightarrow M_d\otimes_{min}(\ell_\infty^{n^2}\oplus_\infty \ell_\infty^{m^2})$ is defined as a direct sum of two maps, it suffices to see that each of these maps $\tilde{J}_\infty^1:M_n\rightarrow M_d(\ell_\infty^{n^2})$ and $\tilde{J}_\infty^2:M_m\rightarrow M_d(\ell_\infty^{m^2})$ are completely contractive respectively. This is trivial since both of them are completely positive and unital\footnote{This second proof, although more stilted, will make the interpolation argument below easier.}. Therefore,
\begin{align}\label{General proof int p infty}
\big\|\tilde{J}_\infty:M_n\oplus_\infty M_m\rightarrow M_d\otimes_{min}(\ell_\infty^{n^2}\oplus_\infty \ell_\infty^{m^2})\big\|_{cb}=1.
\end{align}
The general case (\ref{p-embedding direct sum}) for $1< p<\infty$ follows now by interpolation. However, in this case we need to use a more general result, since we must also interpolate the operators $\tilde{J}_p$. To this end, we can apply Theorem \ref{Thm general interpolation} with $$\tilde{J}_z:=(n^{1-z}\tilde{P}^n\otimes \uno_d)\big(\rho_1\otimes |\psi_d\rangle\langle\psi_d|\big)\oplus (m^{1-z}\tilde{P}^m\otimes \uno_d)\big(\rho_2\otimes |\psi_d\rangle\langle\psi_d|\big).$$In fact, since the theorem is stated for the norm of operators, in order to obtain our estimate for the completely bounded norm, we must consider the family of operators $T_z=id_{M_k}\otimes \tilde{J}_z$ for an arbitrary but fixed $k$. Then, we must understand (\ref{General proof int p=1}) and (\ref{General proof int p infty}) as estimates about the norm of $id_{M_k}\otimes \tilde{J}_1$ and $id_{M_k}\otimes \tilde{J}_\infty$ respectively. On the one hand, according to our explanation in Section \ref{!
 Sec: Basi
 c notions}, we can indeed obtain the spaces $M_k(S_p^n\oplus_p S_p^m)$ and $M_k\big(M_d\otimes_{min}(\ell_p^{n^2}\oplus_p \ell_p^{m^2})\big)$ by interpolating the spaces involved in the estimates (\ref{General proof int p=1}) and (\ref{General proof int p infty}) when they are tensored with $M_k$. On the other hand, since all the spaces are finite dimensional  and the dependence of $T_z$ with respect to $z$ is so simple, all regularity conditions of Theorem \ref{Thm general interpolation} are trivially verified and we just need to see that  $\sup_t\|T_{j+it}\|\leq 1$ for $j=0,1$. Let us recall that  $\tilde{J}_z$ is a direct sum of two maps $\tilde{J}_z^1$ and $\tilde{J}_z^2$. Then, we see that $T_{it}=id_{M_k}\otimes (n^{-it}\tilde{J}_\infty^1\oplus m^{-it}\tilde{J}_\infty^2)$ and similarly $T_{1+it}=id_{M_k}\otimes (n^{-it}\tilde{J}_1^1\oplus m^{-it}\tilde{J}_1^2)$. However, it is very easy to see that the arguments in (\ref{General proof int p=1}) and (\ref{General proof!
  int p in
 fty}) are not affected if we multiply $\tilde{J}_z^1$ and $\tilde{J}_z^2$ by a number of modulus one. Therefore, the same estimates hold in this new case. Hence, we obtain (\ref{p-embedding direct sum}).

Let us consider now the linear map $\Gamma_{p,q}:S_q^d(\ell_\infty^{n^2}\oplus\ell_\infty^{m^2})\rightarrow M_n \oplus M_m$ defined by
\begin{align*}
\Gamma_{p,q}=\frac{1}{d^{1-\frac{1}{q}}}\Big(\frac{1}{n^\frac{1}{p'}}\tilde{W}^n\oplus \frac{1}{m^\frac{1}{p'}}\tilde{W}^m\Big),
\end{align*}where $\tilde{W}^n:M_d\otimes \ell_\infty^{n^2}\rightarrow M_n$ is defined by
\begin{align*}
\tilde{W}^n\Big(\sum_{k,l=1}^nA_{k,l}\otimes e_{k,l}\Big)=\sum_{k,l=1}^n T_{k,l}^*\big((id_n\otimes tr_{n_1})(A_{k,l})\big)T_{k,l},
\end{align*}and $\tilde{W}^m:M_d\otimes \ell_\infty^{m^2}\rightarrow M_m$ is defined analogously. It is clear that $\Gamma_{p,q}$ is completely positive.
We claim that
\begin{align}\label{claim W direct sum}
\big\|\tilde{W}^n:\ell_p^{n^2}(S_p^d)\rightarrow S_p^n\big\|_{cb}\leq (dn)^\frac{1}{p'} \text{    }  \text{    and     }\text{ }\big\|\tilde{W}^m:\ell_p^{m^2}(S_p^d)\rightarrow S_p^m\big \|_{cb}\leq (dm)^\frac{1}{p'}
\end{align}for ever $1\leq p\leq \infty$.
We show the estimate for $\tilde{W}^n$ since the second one is completely analogous.
Let us first consider $p=1$. Then, $\big \|\tilde{W}^n:\ell_1^{n^2}(S_1^d)\rightarrow S_1^n\big \|_{cb}\leq 1$ follows from the fact that $\tilde{W}^n$ is completely positive and trace preserving. On the other hand, the case $p=\infty$ follows from the estimate
\begin{align*}
\big \|\tilde{W}^n:\ell_\infty^{n^2}(M_d)\rightarrow M_n\big \|_{cb}=\Big \|\tilde{W}^n\Big(\sum_{k,l=1}^ne_{k,l}\otimes \uno_d\Big)\Big \|_{M_n}=dn.
\end{align*}
The general estimate (\ref{claim W direct sum}) can be obtained now by interpolation.

With (\ref{claim W direct sum}) at hand, one can show that
\begin{align}
\big\|\Gamma_{p,q}:S_q^d(\ell_p^{n^2}\oplus_p\ell_p^{m^2})\rightarrow S_p^n \oplus_p S_p^m\big\|_{cb}\leq 1.
\end{align}To this end, we use once more that
\begin{align*}
\big\|\Gamma_{p,q}:S_q^d(\ell_p^{n^2}\oplus_p\ell_p^{m^2})\rightarrow S_p^n \oplus_p S_p^m\big\|_{cb}\leq d^{\frac{1}{p}-\frac{1}{q}}\big\|\Gamma_{p,q}:S_p^d(\ell_p^{n^2}\oplus_p\ell_p^{m^2})\rightarrow S_p^n \oplus_p S_p^m\big\|_{cb} .
\end{align*}Therefore, we need to show that
\begin{align*}
\big\|d^{\frac{1}{p}-\frac{1}{q}}\Gamma_{p,q}:S_p^d(\ell_p^{n^2}\oplus_p\ell_p^{m^2})\rightarrow S_p^n \oplus_p S_p^m\big\|_{cb} \leq 1.
\end{align*}Since $d^{\frac{1}{p}-\frac{1}{q}}\Gamma_{p,q}$ does not depend on $q$, let us denote it by $\Gamma_p$.  Now, noting that
\begin{align*}
\Gamma_p=\frac{1}{(nd)^\frac{1}{p'}}\tilde{W}^n\oplus \frac{1}{(dm)^\frac{1}{p'}}\tilde{W}^m,
\end{align*}the previous estimate is a direct consequence of (\ref{claim W direct sum}).

Therefore, we conclude our proof if we show that
\begin{align*}
\Gamma_{p,q} \circ \tilde{J}_{p,q}=\id_{S_p^n\oplus_p S_p^m}.
\end{align*}Indeed, given $\rho_1\oplus \rho_2\in S_p^n\oplus_p S_p^m$ we have that
\begin{align*}
\Gamma_p\big(\tilde{J}_p(\rho_1\oplus \rho_2)\big)&=d^{-\frac{1}{q}}\Gamma_p\Big(\frac{1}{n^\frac{1}{p}}\sum_{j=1}^{n_1}\sum_{k,l=1}^ne_{k,l}\otimes T_{k,l}\rho_1 T_{k,l}^*\otimes |e_j\rangle\langle e_{j}|\\&\text{    }\text{    }\text{    }\oplus \frac{1}{m^\frac{1}{p}}\sum_{i=1}^{m_1}\sum_{k',l'=1}^me_{k',l'}\otimes T_{k',l'}\rho_2 T_{k',l'}^*\otimes |e_i\rangle\langle e_i|\Big)\\&=
\frac{1}{n^\frac{1}{p}}\frac{1}{dn^\frac{1}{p'}}n_1\sum_{k,l=1}^n\rho_1 \oplus \frac{1}{m^\frac{1}{p}}\frac{1}{dm^\frac{1}{p'}}m_1\sum_{k',l'=1}^m\rho_2 =\rho_1\oplus \rho_2.
\end{align*}
\end{proof}
%
%
%\begin{corollary}
%Let $1\leq p,q\leq \infty$. Let $n_1,\cdots ,n_k$ be a family of natural numbers such that $n_1\geq n_i$ for every $i$. There exists a completely positive and completely isometric embedding$$\hat{J}_{p,q}:S_p^{n_1}\oplus_p \cdots \oplus_p S_p^{n_k}\rightarrow S_q^{n_1}(\ell_p^{kn_1^2}).$$Moreover, there exists a completely positive and completely contractive projection $\hat{W}_{p,q}:S_q^{n_1}(\ell_p^{kn_1^2}) \rightarrow S_p^{n_1}\oplus_p \cdots \oplus_p S_p^{n_k}$ onto the image of  $\hat{J}_{p,q}$.
%\end{corollary}
%\begin{proof}
%We just need to note that there is a completely positive and completely contractive embedding of $S_p^{n_1}\oplus_p \cdots \oplus_p S_p^{n_k}$ in $S_p^{n_1}\oplus_p \cdots \oplus_p S_p^{n_1}$, whose image is complemented by a completely positive and completely contractive projection. Then, the result follows straightforward from Theorem \ref{Teleport embedding general}
%\end{proof}
%
%
%
We finish this section by proving Theorem \ref{Theorem superdense embedding}. The ideas here are motivated by another communication protocol called super dense coding, in which Alice can send 2 bits of classical communication to Bob by just send 1 qubit of communication if they are allowed to share a maximally entangled state during the protocol.
\begin{prop}\label{prop: p embedding superdense}
Let us define the linear map $H:\ell_\infty^{n^2}\rightarrow M_n\otimes M_n$ by
\begin{align*}
H(e_{k,l})=n|\eta_{k,l}\rangle\langle\eta_{k,l}|
\end{align*}for every $k,l$. Then, $H$ is completely positive and it verifies, for every operator space $E$,
\begin{align*}
\|H\otimes id_E:\ell_p^{n^2} [E]\rightarrow S_q^n(S_p^n[E])\|_{cb}\leq n^{1+\frac{1}{q}-\frac{1}{p}}
\end{align*}for every $1\leq p\leq q \leq\infty$.
\end{prop}
\begin{proof}
Since the domain space is a commutative C$^*$-algebra, completely positivity is equivalent to positivity. Hence, the fact that $n|\eta_{k,l}\rangle\langle\eta_{k,l}|$ is a positive element for every $k,l$ assures that $H$ is indeed completely positive. On the other hand, we have already explained that
\begin{align*}
\big\|H\otimes id_E:\ell_p^{n^2}[E] \rightarrow S_q^n(S_p^n[E])\big\|_{cb}\leq n^\frac{1}{q}
\big\|H\otimes id_E:\ell_p^{n^2}[E] \rightarrow M_n(S_p^n[E])\big\|_{cb},
\end{align*}so we must show the estimate
\begin{align}\label{p-embedding superdense}
\big\|H\otimes id_E:\ell_p^{n^2}[E] \rightarrow M_n(S_p^n[E])\big\|_{cb}\leq n^{1-\frac{1}{p}}
\end{align}for every operator space $E$. In order to show this estimate let us start with the case $p=1$,
\begin{align}\label{p-embedding superdense p=1}
\big\|H\otimes id_E:\ell_1^{n^2} [E]\rightarrow M_n(S_1^n[E])\big \|_{cb}\leq 1.
\end{align}Since $\ell_1^{n^2}$ is a maximal operator space (see \cite[Chapter 3]{Pisierbook}), we have that $$\|H:\ell_1^{n^2} \rightarrow M_n(S_1^n)\|_{cb}=\|H:\ell_1^{n^2} \rightarrow M_n(S_1^n)\|.$$ Furthermore, by a convexity argument one can easily deduce that $\|H\|=\sup_{k,l}\|H(e_{k,l})\|_{M_n(S_1^n)}$. Now, by noting that $$H(e_{k,l})=n|\eta_{k,l}\rangle\langle\eta_{k,l}|=\sum_{i,j=1}^nu_ke_{i,j}u_k^*\otimes v_le_{i,j}v_l^*,$$and recalling that $\|\sum_{i,j=1}^ne_{i,j}\otimes e_{i,j}\|_{M_n(S_1^n)}=1$ (see the proof of Lemma \ref{lemma: assist-ent estimate}), it is very easy to conclude that $\|H(e_{k,l})\|_{M_n(S_1^n)}=1$ for every $k,l$. On the other hand, according to (\ref{mmp projective norm}) the previous estimate implies that
\begin{align*}
\big\|H\otimes id_E:\ell_1^{n^2} [E]\rightarrow M_n(S_1^n)\hat{\otimes}E\big \|_{cb}\leq 1.
\end{align*}Hence, (\ref{p-embedding superdense p=1}) follows from the fact that $\big\|id:M_n(S_1^n)\hat{\otimes}E\rightarrow M_n(S_1^n[E])\big \|_{cb}\leq 1$, which can be obtained from the definition of the projective tensor norm.

In order to prove the estimate for $p=\infty$ we just note that
\begin{align*}
\big\|H:\ell_\infty^{n^2} \rightarrow M_{n^2}\big\|_{cb}=\big\|H(1)\big\|_{M_{n^2}}=\big \|n\uno_{n^2}\big\|_{M_{n^2}}=n.
\end{align*}According to (\ref{mmp min norm}), this implies that
\begin{align}\label{p-embedding superdense p=infty}
\big\|H\otimes id_E:\ell_\infty^{n^2}[E] \rightarrow M_n(M_n[E])\big\|_{cb}=n.
\end{align}
The estimate (\ref{p-embedding superdense}) for the general case $1<p<\infty$ can be now deduced from (\ref{p-embedding superdense p=1}), (\ref{p-embedding superdense p=infty}) and a standard interpolation argument (\ref{interpolation step I}).
\end{proof}
\begin{prop}\label{prop:p-projection superdense}
Let $Q:M_n\otimes M_n\rightarrow \ell_\infty^{n^2}$ be the linear map defined by
\begin{align*}
Q(\rho)=\frac{1}{n}\sum_{k,l=1}^n \langle\eta_{k,l}|\rho|\eta_{k,l}\rangle e_{k,l}, \text{    } \text{    } \rho\in M_{n^2}.
\end{align*}Then, $Q$ is completely positive and it verifies, for every operator space $E$,
\begin{align}\label{p-projection superdense}
\big\|Q\otimes id_E:S_q^n(S_p^n[E])\rightarrow \ell_p^{n^2}[E]\big\|_{cb}\leq n^{\frac{1}{p}-\frac{1}{q}-1},
\end{align}for every $1\leq p \leq q \leq\infty$.
\end{prop}
\begin{proof}
Note that $Q=\frac{1}{n}P$, where $P$ was introduced in Corollary \ref{orth-complementation}. Therefore, the statement of the proposition is clear just noting that
\begin{align*}
\big\|Q\otimes id_E: S_q^n(S_p^n[E])\rightarrow \ell_p^{n^2}[E]\big\|_{cb}\leq n^{\frac{1}{p}-\frac{1}{q}}\frac{1}{n}\big\|P\otimes id_E:S_p^n(S_p^n[E])\rightarrow \ell_p^{n^2}[E]\big\|_{cb}=n^{\frac{1}{p}-\frac{1}{q}-1}.
\end{align*}
\end{proof}
\begin{proof}[Proof of Theorem \ref{Theorem superdense embedding}]
Again, by duality it suffices to consider the case Let $1\leq p\leq q\leq \infty$.

Let us define the linear maps
\begin{align*}
H_{p,q}:=n^{\frac{1}{p}-1-\frac{1}{q}}H:\ell_\infty^{n^2} \rightarrow M_n\otimes M_n,
\end{align*}where $H$ was defined in Proposition \ref{prop: p embedding superdense}, and
\begin{align*}
Q_{p,q}:=n^{1-\frac{1}{p}+\frac{1}{q}}Q:M_n\otimes M_n\otimes M_n\rightarrow \ell_\infty^{n^2},
\end{align*}where $Q$ was defined in Proposition \ref{prop:p-projection superdense}. According to Proposition \ref{prop: p embedding superdense} and Proposition \ref{prop:p-projection superdense}, both maps are completely positive and they verify the following estimates:
\begin{align*}
\big\|H_{p,q}\otimes id_E:\ell_p^{n^2}[E] \rightarrow S_q^n(S_p^n[E])\big\|_{cb}\leq 1,\text{    }\text{    } \text{  and   } \text{    } \text{    } \big\|Q_{p,q}\otimes id_E:S_q^n(S_p^n[E])\rightarrow \ell_p^{n^2}[E]\big\|_{cb}\leq 1.
\end{align*}Therefore, it suffices to show the algebraic identification $Q_{p,q} \circ H_{p,q}=id_{\ell_\infty^{n^2}}$.
This is very easy by noting that for every $e_{k,l}\in \ell_\infty^{n^2}$
\begin{align*}
Q_{p,q} \big(H_{p,q}(e_{k,l})\big)=Q(H(e_{k,l}))=e_{k,l}.
\end{align*}
\end{proof}
\section{Some results about covariant channels}\label{Sec: Some results about covariant channels}
In this section we will introduce a nice family of channels and we will explain why computing some capacities of these channels is easier than in the general case. First, let us recall that a state (or density operator) $\rho$ is a positive operator (acting on Hilbert spaces) with trace equal one. In fact, in this work we will restrict to finite dimensional Hilbert spaces, so a state (or density matrix) is a semidefinite positive matrix $\rho\in M_n$ such that $\tr(\rho)=1$. We will write $\rho\in S_1^n$ to denote a general state. In fact, very often we will consider bipartite states, which means that $\rho$ is a state acting on the tensor product of two Hilbert spaces, say $\ell_2^d\otimes_2 \ell_2^n$. In this case, we will denote $\rho\in S_1^d\otimes S_1^n=S_1^{dn}$. We will say that $\rho$ is a pure state if it is a rank one projection $\rho=|\psi\rangle\langle \psi|$ onto a unit vector $\psi\in \ell_2^n$. To be consistent with the standard notation in quantum informatio!
 n, we wil
 l write $|\psi\rangle \in \C^n$ to denote one of these unit vectors\footnote{Ket-notation $|\psi\rangle$ denotes a general unit element in a Hilbert space, while bra-notation $\langle \psi |$ is used to denote it as a dual element.}. Then, a general pure bipartite state will be described by $\rho=|\psi\rangle\langle \psi|$ with $|\psi\rangle\in \C^d\otimes \C^n=\C^{dn}$. We will also make use of a very important quantity in quantum information called von Neumann entropy. Given a state $\rho$, its von Neumann entropy is defined as $$S(\rho)=-tr\big(\rho \log_2\rho\big).$$This is a generalization of the Shannon entropy of a probability distribution already introduced in Theorem \ref{Main Theorem Erasure}. We start this section by recalling the following well known result, which can be found in \cite{AHW}.
\begin{lemma}\label{derivate-Entropy}
The function $F(\rho,p)= \frac{1-\|\rho\|_p}{p-1}$ is well defined for $p$ positive with $p\neq 1$ and $\rho$ a density matrix. It can be extended by continuity to $p\in (0,\infty)$ and this extension verifies $$F(\rho,1)=-\frac{d}{dp}\|\rho\|_p\big|_{p=1}=S(\rho).$$Moreover, the convergence at $p=1$ is uniform in the states $\rho$.

In particular, for every net $(\rho_p)_p$ of states such that  $\lim_{p  \rightarrow 1}\rho_p= \rho$ in the trace class norm, we have that $\lim_{p  \rightarrow 1}F(\rho_p,p)=S(\rho)$.
\end{lemma}
Indeed, although the first part of the result was proved in  \cite{AHW} for the function  $\frac{1-\|\rho\|_p^p}{p-1}$, it is very easy to conclude that, then, the same result must hold for
the function $F(\rho,p)$.
%Indeed, just note that
%\begin{align*}
%\Big\ba \frac{1-\|\rho\|_p}{p-1}-S(\rho)\Big\ba\leq \Big\ba \frac{1-\|\rho\|_p}{p-1}-\frac{1-\|\rho\|^p_p}{p-1}\Big\ba + \Big\ba \frac{1-\|\rho\|^p_p}{p-1}-S(\rho)\Big \ba.
%\end{align*}Now, we are assuming that the second term goes to zero uniformly. On the other hand one can also control the first term by  noting that
%\begin{align*}
%\Big\ba \frac{\|\rho\|^p_p-\|\rho\|_p}{p-1}\Big\ba= \frac{\|\rho\|_p-\|\rho\|^p_p}{p-1}\leq -\|\rho\|_p\log (\|\rho\|_p)\leq -\log (\|\rho\|_p)\leq \big(1-\frac{1}{p}\big)\log n,
%\end{align*}which tends to zero as $p$ tends to 1 uniformly in $\rho$. Note that in the first inequality we have used  that for any real numbers $\lambda\in (0,1]$ and $p\geq 1$ we $\lambda- \lambda^p\leq \lambda(-\log \lambda)(p-1)$ (see Remark 3.2 in \cite{JuPa}) and in the last inequality we have used that $\|id: S_p^n\rightarrow S_1^n\|=n^{1-\frac{1}{p}}$.
On the other hand, the second part of the statement is a direct consequence of the uniform convergence and the continuity of the von Neuman entropy (see for instance \cite{Aud1}):
\begin{theorem}\label{cont-Entropy}
For all $n$-dimensional states $\rho$, $\sigma$ we have
\begin{align*}
|S(\rho)-S(\sigma)|\leq T\log (n-1)+ H((T,1-T)),
\end{align*}where $T=\frac{\|\rho-\sigma\|_1}{2}$ and $H$ denotes the Shannon entropy.
\end{theorem}
%
%
%Indeed, if $\lim_{p  \rightarrow 1}\rho_p= \rho$ we have
%\begin{align*}
%\big|F(\rho_p,p)-S(\rho)\big|\leq \big|F(\rho_p,p)-S(\rho_p)\big|+\big|S(\rho_p)-S(\rho)\big|.
%\end{align*}Then, the first term tends to zero because of the uniform convergence of the function  $F$ in the states $\rho$ and the second one tends to zero because of Theorem \ref{cont-Entropy}.
Lemma \ref{derivate-Entropy} has motivated the study of channel capacities by means of the derivative of certain $p$-norms defined on these channels (see for instance \cite{AHW} and \cite{DJKR}). More precisely, since a quantum channel $\Ne$ is nothing else than a completely positive and trace preserving map from $M_n$ to $M_m$ (we will denote it by $\Ne:S_1^n\rightarrow S_1^m$) one can consider (and differentiate) de function $f(p)=\|\Ne:S_1^n\rightarrow S_p^m\|$. Indeed, the quantity $\frac{d}{dp}f(p) \ba_{p=1}$ has been shown to be related to the (product state) classical capacity, also called Holevo capacity,  of the quantum channel $\Ne$. However, in the recent paper \cite{JuPa2} the authors showed that, in order to exactly describe the (product state) classical capacity of a quantum channel with $d$-assisted entanglement, $C_{prod}^d(\Ne)$, as a derivative of a function, one has to consider the completely $\ell_q(S_q^d)$-summing  norm of the channel. Formally, one has !
 the follo
 wing result.
\begin{theorem}\label{mainJP2}
Given a quantum channel $\Ne:S_1^n\rightarrow S_1^m$ and a natural number $d$ verifying $1\leq d\leq n$, we find
\begin{equation*}
C_{prod}^d(\Ne)= \frac{d}{dp}\big[\pi_{q,d}(\Ne^*)\big]|_{p=1},
\end{equation*}where $\frac{1}{p}+\frac{1}{q}=1$. Here, $\pi_{q,d}(\Ne^*)$ denotes the $\ell_q(S_q^d)$-summing norm of $\Ne^*:M_m\rightarrow M_n$.
\end{theorem}
\begin{remark}\label{log 2 vs lgn}
Actually, to have the equality in the previous theorem we must define $C_{prod}^d(\Ne)$ (\cite[Equation (1.3)]{JuPa2}) by using the \emph{$\ln$-entropy}, $S(\rho):= -tr(\rho \ln \rho)$, instead of using $\log_2$ as it is usually done in quantum information. Since both definitions are the same up to a multiplicative factor, we can use the standard entropy $S$ and we must then write the previous expression as $C_{prod}^d(\Ne)= \frac{1}{\ln 2}\frac{d}{dp}\big[\pi_{q,d}(\Ne^*)\big]|_{p=1}$. In order to avoid the $\ln 2$ term in all our statements, we will still consider here the definition of $C_{prod}^d(\Ne)$ as in the previous work \cite{JuPa2}. However, in order to state our results in Theorem \ref{Main theorem Depolarizing} and Theorem \ref{Main Theorem Erasure} (where we want to consider the standard definitions in quantum information theory) we will need to multiply our results by $\frac{1}{\ln 2}$. As the reader will see, this will be only reflected in replacing $\ln$ by !
 $\log_2$
 and $\ln$-entropies by $\log_2$-entropies, since these are the only terms appearing in our main statements.
\end{remark}
In many cases, the factorization associated to the  $\ell_q(S_q^d)$-summing norm of $\Ne^*$ has a particularly nice form. This is the case of covariant channels where one can show that
\begin{align}\label{covariant single}
C_{prod}^d(\Ne)=\ln  n+ \frac{d}{dp}\big\|\Ne:S_1^n\rightarrow S_p^m\big\|_d|_{p=1},
\end{align}where here $\big\|\Ne:S_1^n\rightarrow S_p^m\big\|_d$ denotes the $d$-norm: $\big\|id_d\otimes \Ne:M_d(S_1^n)\rightarrow M_d(S_p^m)\big\|$.

In this work we will mainly deal with covariant channels. The next result shows that one can restrict to pure states in the computation of this quantity.
\begin{theorem}\label{theorem:p-norm pure states}
Given a quantum channel $\Ne:S_1^n\rightarrow S_1^m$ and $1\leq d\leq n$, let us define the quantity $$S_d(\Ne):=\frac{d}{dp}\big\|\Ne:S_1^n\rightarrow S_p^m\big\|_d|_{p=1}.$$ Then,
\begin{align*}
S_d(\Ne)=\sup \Big\{S\big(id_d\otimes tr_n)(|\psi\rangle\langle \psi|)\big)-S\big(id_d\otimes \Ne)(|\psi\rangle\langle \psi|)\big)\Big\}
\end{align*}where the supremum is taking over all unit vectors $|\psi\rangle\in \C^d\otimes \C^n$.
\end{theorem}
The quantity $S_d(\Ne)$ is a generalization of the cb-min entropy introduced in \cite{DJKR}. In particular, the quantity cb-min corresponds to $S_n(\Ne)$.
\begin{proof}
According to (\ref{CB norm with general $t$}) we have
\begin{align*}
\frac{d}{dp}\big\|\Ne:S_1^n\rightarrow S_p^m\big\|_d|_{p=1}&=\frac{d}{dp}\big\|id_d\otimes \Ne:S_p^d(S_1^n)\rightarrow S_p^{dm}\big\||_{p=1}\\&\geq\lim_{p\rightarrow 1}\sup_{\rho\in S_1^{dn}}\frac{1}{\big\|\rho\big\|_{S_p^d(S_1^n)}}\frac{\big\|(id_d\otimes \Ne)(\rho)\big\|_{S_p^{dm}}-\big\|\rho\big\|_{S_p^d(S_1^n)}}{p-1}\\&=\lim_{p\rightarrow 1}\sup_{\rho\in S_1^{dn}}\frac{1}{\big\|\rho\big\|_{S_p^d(S_1^n)}}\Big(\frac{\big\|(id_d\otimes \Ne)(\rho)\big\|_{S_p^{dm}}-1}{p-1}+\frac{1-\big\|\rho\big\|_{S_p^d(S_1^n)}}{p-1}\Big)\\&=\sup_{\rho\in S_1^{dn}} \Big\{S\big(id_d\otimes tr_n)(\rho)\big)-S\big(id_d\otimes \Ne)(\rho)\big)\Big\}.
\end{align*}Here, the first inequality is due to the fact that we are restricting the computation of the norm to states $\rho\in S_1^{dn}$ rather than to general matrices $\rho\in M_{dn}$. We have also used  that, by Lemma \ref{covariant single} and the fact that $\big\|\rho\big\|_{S_p^d[S_1^n]}=\big\|(id_d\otimes tr_n)(\rho)\big\|_{S_p^d}$ for positive elements (see Remark \ref{(p,1)}), we have that $$\lim_{p\rightarrow 1}\frac{\big\|(id_d\otimes \Ne)(\rho)\big\|_{S_p^{dn}}-1}{p-1}=-S\big((id_d\otimes \Ne)(\rho)\big)$$ and $$\lim_{p\rightarrow 1}\frac{1-\big\|\rho\big\|_{S_p^d(S_1^n)}}{p-1}=\lim_{p\rightarrow 1}\frac{1-\big\|(id_d\otimes tr_n)(\rho)\big\|_{S_p^d}}{p-1}=S\big((id_d\otimes tr_n)(\rho)\big)$$ uniformly. Therefore, we can iterate the limite and the supremum.

On the other hand, according to (\ref{CB norm with general $t$}) we also have
\begin{align*}
\big\|\Ne:S_1^n\rightarrow S_p^m\big\|_d=\big\|id_d\otimes \Ne:S_1^{dn}\rightarrow S_1^d(S_p^m)\big\|=\sup_{|\psi\rangle\in \C^{dn}}\big\|(id_d\otimes \Ne)(|\psi\rangle\langle \psi|)\big\|_{S_1^d(S_p^m)}.
\end{align*}Here, we have used that, since $\Ne$ is completely positive, we can compute its completely bounded norm by restricting to positive elements (Remark \ref{(p,1)}). Then, by normalizing we can restrict to states. Furthermore, since pure states are exactly the extreme points of the set of states, we have the last equality. Then,
\begin{align*}
\frac{d}{dp}\big\|\Ne:S_1^n\rightarrow S_p^n\big\|_d|_{p=1}&=\lim_{p\rightarrow 1}\sup_{|\psi\rangle\in \C^{dn}}\frac{\big\|(id_d\otimes \Ne)(|\psi\rangle\langle \psi|)\big\|_{S_1^d(S_p^m)}-1}{p-1}\\&\leq
\lim_{p\rightarrow 1}\sup_{|\psi\rangle\in \C^{dn}}\frac{tr_d\Big(\big(id_d\otimes tr_m\big)\big((id_d\otimes \Ne)(|\psi\rangle\langle \psi|)\big)^p\Big)^\frac{1}{p}-1}{p-1},
\end{align*}where here we have used that for ever positive element $x\in M_d\otimes M_m$ we have (see \cite{KiKo})
\begin{align*}
\|x\|_{S_1^d(S_p^m)}\leq \Big\|\big((id_d\otimes tr_m)(x^p)\big)^\frac{1}{p}\Big\|_{S_1^d}.
\end{align*}Let us call for a fixed $|\psi\rangle\in \C^{dn}$, $\rho_\psi=(id_d\otimes \Ne)(|\psi\rangle\langle \psi|)$ and note that
\begin{align*}
\frac{tr_d\Big((id_d\otimes tr_m)(\rho_\psi^p)\Big)^\frac{1}{p}-1}{p-1}=\frac{tr_d\big((id_d\otimes tr_m)(\rho_\psi^p)\big)^\frac{1}{p}-(tr_d\otimes tr_m)(\rho_\psi^p)}{p-1}+\frac{(tr_d\otimes tr_m)(\rho_\psi^p)-1}{p-1}\\=
\frac{tr_d\Big[\big((id_d\otimes tr_m)(\rho_\psi^p)\big)^\frac{1}{p}-(id_d\otimes tr_m)(\rho_\psi^p)\Big]}{p-1}+\frac{(tr_d\otimes tr_m)(\rho_\psi^p)-1}{p-1}\\\leq -tr_d\Big[(id_d\otimes tr_m)(\rho_\psi^p)\big)^\frac{1}{p} \ln \Big((id_d\otimes tr_m)(\rho_\psi^p)\big)^\frac{1}{p}\Big)\Big]+\frac{(tr_d\otimes tr_m)(\rho_\psi^p)-1}{p-1}.
\end{align*}Here we have used functional calculus and Remark 3.2 in \cite{JuPa2}. Now, it is not difficult to see that the function
\begin{align}\label{function G}
G(p,\rho)=-tr_d\Big[(id_d\otimes tr_m)(\rho^p)\big)^\frac{1}{p} \ln \Big(\big((id_d\otimes tr_m)(\rho^p)\big)^\frac{1}{p}\big)\Big)\Big]+\frac{(tr_d\otimes tr_n)(\rho^p)-1}{p-1}
\end{align}verifies that
\begin{align*}
\lim_{p\rightarrow 1}G(p,\rho)=S\big((id_d\otimes tr_m)(\rho))\big)- S(\rho)
\end{align*}and that this convergence is uniform in the states $\rho\in S_1^{dm}$. Indeed, the uniform convergence for the second term in (\ref{function G}) is a direct consequence of Lemma \ref{derivate-Entropy}. On the other hand, the uniform convergence of the first term in (\ref{function G}) can be easily obtained from Theorem \ref{cont-Entropy}.

Hence, we can finish our proof by using (\ref{function G}) and noting that
\begin{align*}
\frac{d}{dp}\|\Ne:S_1^n\rightarrow S_p^n\|_d|_{p=1}&\leq \lim_{p\rightarrow 1}\sup_{|\psi\rangle\in \C^{dn}}G(p,\rho_\psi)=\sup_{|\psi\rangle\in \C^{dn}}\Big\{S\Big((id_d\otimes tr_n)(\rho_\psi)\Big)- S(\rho_\psi)\Big\}\\&=
\sup_{|\psi\rangle\in \C^{dn}}\Big\{S\Big((id_d\otimes tr_n)(|\psi\rangle\langle \psi|)\Big)- S\Big((id_d\otimes \Ne)(|\psi\rangle\langle \psi|)\Big)\Big\}.
\end{align*}
\end{proof}
In this work, we are interested in dealing with quantum channels of the form
\begin{align}\label{channel direct sum}
\Ne:S_1^n\rightarrow S_1^{n_1}\oplus_1\cdots \oplus _1S_1^{n_m}
\end{align}
such that $$\Ne(\rho)=\mu_1\Ne_1(\rho)\oplus\cdots\oplus \mu_m\Ne_m(\rho),$$where $(\mu_j)_{j=1}^m$ is a probability distribution and $\Ne_j:S_1^n\rightarrow S_1^{n_j}$ is a quantum channel for every $j$.
\begin{definition}\label{Def:covariant}
Let $G$ be a compact group and let us consider unitary representations $\pi:G\rightarrow \mathbb U(n)$  and $\sigma_j:G\rightarrow \mathbb U(n_j)$ for every $j=1,\cdots, m$. We say that a quantum channel $\Ne$  of the form (\ref{channel direct sum}) is \emph{covariant} (with respect to ($G,\pi,\sigma_1,\cdots,\sigma_m$) if
\begin{enumerate}
\item[1.] $\int_G \sigma_j (g)^*\rho \sigma_j(g)dg=\frac{tr(\rho)}{n_j}\uno_{n_j}$ for every  $\rho\in S_1^{n_j}$ and for every $j$. Here, $ \mathbb U(n_j)$ represents the unitary group in dimension $n_j$ and the integral is with respect to the Haar measure  of $G$.
\item[2.]  $\Ne_j\big(\pi(g)^*\rho\pi(g)\big)=\sigma_j(g)^*\Ne_j(\rho)\sigma_j(g)$ for every $g\in G$ and every $\rho\in S_1^n$.
\end{enumerate}
\end{definition}
\begin{prop}\label{prop:d-restricted capacity direct sume}
Given a quantum channel $\Ne:S_1^n\rightarrow S_1^{n_1}\oplus_1\cdots \oplus_1 S_1^{n_m}$ as in (\ref{channel direct sum}), we have
\begin{align}\label{d-restricted capacity direct sume}
C^d_{prod}(\Ne)=\sup\sum_{j=1}^m\mu_j\Big\{S\Big(\sum_{i=1}^N \lambda_i \Ne_j\big((tr_d\otimes id_n)(\rho_i)\big)\Big)
\\ \nonumber+\sum_{i=1}^N \lambda_i\Big[S\Big((id_d\otimes tr_n)(\rho_i)\Big)-S\Big(\big(id_d\otimes \Ne_j\big)(\rho_i)\Big)\Big] \Big\}.
\end{align}Here, the supremum runs over all $N\in \N$, all probability distributions $(\lambda_i)_{i=1}^N$ and all families $(\rho_i)_{i=1}^N$, where $\rho_i\in S_1^d\otimes S_1^n$ is a state for every $i=1,\cdots, N$.
\end{prop}
\begin{proof}
According to \cite[Proposition 5.5]{JuPa2}, for a channel $\Ne:S_1^n\rightarrow S_1^m$, we have that
\begin{align}\label{d-restricted capacity}
C^d_{prod}(\Ne)=\sup\Big\{S\Big(\sum_{i=1}^N \lambda_i \Ne\big((tr_d\otimes id_n)(\rho_i)\big)\Big)
\\ \nonumber+\sum_{i=1}^N \lambda_i\Big[S\Big((id_d\otimes tr_n)(\rho_i)\Big)-S\Big(\big(id_d\otimes \Ne\big)(\rho_i)\Big)\Big] \Big\}.
\end{align}Here, the supremum runs over all $N\in \N$, all probability distributions $(\lambda_i)_{i=1}^N$, and all families $(\rho_i)_{i=1}^N$, where $\rho_i\in S_1^d\otimes S_1^n$ is a state for every $i=1,\cdots, N$.

However, it is very easy to check that if $\Ne:S_1^n\rightarrow S_1^{n_1}\oplus_1\cdots \oplus_1 S_1^{n_m}\subset S_1^{n_1+\cdots +n_m}$ is as in (\ref{channel direct sum}) we have
\begin{align*}
&S\Big(\sum_{i=1}^N \lambda_i \Ne\big((tr_d\otimes id_n)(\rho_i)\big)\Big)=H\big((\mu_j)_{j=1}^m\big)+\sum_{j=1}^m\mu_jS\Big(\sum_{i=1}^N \lambda_i \Ne_j\big((tr_d\otimes id_n)(\rho_i)\big)\Big),\\& S\Big(\big(id_d\otimes \Ne\big)(\rho_i)\Big)=H\big((\mu_j)_{j=1}^m\big)+\sum_{j=1}^m\mu_jS\Big(\big(id_d\otimes \Ne_j\big)(\rho_i)\Big).
\end{align*}Here, $H\big((\mu_j)_{j=1}^m\big)$ is the Shannon entropy of the probability distribution $(\mu_j)_{j=1}^m$ already introduced in Theorem \ref{Main Theorem Erasure}. Then, the result follows.
\end{proof}
Let us now define, for a channel $\Ne:S_1^n\rightarrow S_1^{n_1}\oplus_1\cdots \oplus_1 S_1^{n_m}$ as in (\ref{channel direct sum}), the quantity
\begin{align}\label{V_d}
V_d(\Ne)=\sup \Big\{\sum_{j=1}^m\mu_j\Big[S\Big((id_d\otimes tr_n)(|\psi\rangle\langle \psi|)\Big)-S\Big((id_d\otimes \Ne_j)(|\psi\rangle\langle \psi|)\Big)\Big]\Big\},
\end{align}where the supremum is taking over all pure states $|\psi\rangle\in \C^d\otimes \C^n$.
\begin{lemma}\label{lemma:V_d vs S_d}
Given a channel $\Ne:S_1^n\rightarrow S_1^{n_1}\oplus_1\cdots \oplus_1 S_1^{n_m}$ as in (\ref{channel direct sum}), we have
\begin{align*}
S_d(\Ne)=V_d(\Ne)- H\big((\mu_j)_{j=1}^n\big).
\end{align*}Furthermore,
\begin{align*}
V_d(\Ne)=\sup \Big\{\sum_{j=1}^m\mu_j\Big[S\big(id_d\otimes tr_n)(\rho)\big)-S\big(id_d\otimes \Ne_j)(\rho)\big)\Big]\Big\},
\end{align*}where the supremum is taking over all states $\rho\in S_1^d\otimes S_1^n$.
\end{lemma}
\begin{proof}
According to Theorem \ref{theorem:p-norm pure states}, we have
\begin{align*}
S_d(\Ne)=\sup \Big\{S\Big((id_d\otimes tr_n)(|\psi\rangle\langle \psi|)\Big)-S\Big((id_d\otimes \Ne)(|\psi\rangle\langle \psi|)\Big)\Big\},
\end{align*}where the supremum is taking over all pure states $|\psi\rangle\in \C^d\otimes \C^n$. On the other hand, it is very easy to see that for every state (pure or not)
\begin{align*}
S\big((id_d\otimes \Ne)(\rho)\big)=H((\mu_j)_{j=1}^n)+\sum_{j=1}^m\mu_jS\big((id_d\otimes \Ne_j)(\rho)\big).
\end{align*}Therefore, the first statement follows.

The second part of the statement follows from the fact that the definition of $S_d(\Ne)$ doesn't change if we take the supremum over all states (see Theorem \ref{theorem:p-norm pure states}).
\end{proof}
In the following proposition we give a nice formula to compute $C^d_{prod}(\Ne)$ for covariant channels.
\begin{prop}\label{prop:C^d_{prod} for covariant direct sume}
Let $\Ne:S_1^n\rightarrow S_1^{n_1}\oplus_1\cdots \oplus_1 S_1^{n_m}$ be a quantum channel as in (\ref{channel direct sum}) which is covariant. Then,
\begin{align*}
C^d_{prod}(\Ne)=\sum_{j=1}^m\mu_j\ln n_j+ V_d(\Ne).
\end{align*}
\end{prop}
\begin{proof}
Since $S\Big(\sum_{i=1}^N \lambda_i \Ne_j\big((tr_d\otimes id_n)(\rho_i)\big)\Big)\leq \ln n_j$ for every $N\in \N$, all probability distributions $(\lambda_i)_{i=1}^N$, and all families $(\rho_i)_{i=1}^N$ of states $\rho_i\in S_1^d\otimes S_1^n$, Proposition \ref{prop:d-restricted capacity direct sume} guarantees that
\begin{align*}
C^d_{prod}(\Ne)&\leq  \sum_{j=1}^m\mu_j \ln n_j+ \sup \Big\{\sum_{j=1}^m\mu_j\sum_{i=1}^N \lambda_i\Big[S\Big((id_d\otimes tr_n)(\rho_i)\Big)-S\Big(\big(id_d\otimes \Ne_j\big)(\rho_i)\Big)\Big] \Big\}\\&=
\sum_{j=1}^m\mu_j \ln n_j+ \sup \Big\{\sum_{i=1}^N \lambda_i\Big[S\Big((id_d\otimes tr_n)(\rho_i)\Big)-\sum_{j=1}^m\mu_jS\Big(\big(id_d\otimes \Ne_j\big)(\rho_i)\Big)\Big] \Big\},
\end{align*}where the supremum runs over all $N\in \N$, all probability distributions $(\lambda_i)_{i=1}^N$, and all families $(\rho_i)_{i=1}^N$, of states $\rho_i\in S_1^d\otimes S_1^n$. Now, by convexity it is clear that this is the same as
\begin{align*}
C^d_{prod}(\Ne)&\leq  \sum_{j=1}^m\mu_j \ln n_j+ \sup \Big\{S\Big((id_d\otimes tr_n)(\rho)\Big)-\sum_{j=1}^m\mu_jS\Big(\big(id_d\otimes \Ne_j\big)(\rho)\Big) \Big\}\\&=
\sum_{j=1}^m\mu_j \ln n_j+ \sup \Big\{\sum_{j=1}^m\mu_j\Big[S\Big((id_d\otimes tr_n)(\rho)\Big)-S\Big(\big(id_d\otimes \Ne_j\big)(\rho)\Big) \Big]\Big\},\\
\end{align*}where the supremum runs over all states $\rho\in S_1^d\otimes S_1^n$. Then, we conclude that
\begin{align*}
C^d_{prod}(\Ne)\leq  \sum_{j=1}^m\mu_j \ln n_j+ V_d(\Ne).
\end{align*}
Let us now consider a general state $\rho\in S_1^d\otimes S_1^n$ (in particular, any pure state). For every $g\in G$ we denote $\rho_g:=\big(id_d\otimes \pi(g)^*\big)\rho \big(id_d\otimes \pi(g)\big)$ and we consider the ensemble $\{dg, (\rho_g)_g\}$\footnote{Although we usually consider finite ensembles $\{(\lambda_i)_{i=1}^N, (\rho_i)_{i=1}^N\}$ one can also work with infinite ones and obtain the corresponding result by approximation.}. Then, according to Proposition \ref{prop:d-restricted capacity direct sume} we have
\begin{align*}
C^d_{prod}(\Ne)\geq \sum_{j=1}^m\mu_j\Big\{S\Big(\int_G \Ne_j\big((tr_d\otimes id_n)(\rho_g)\big)dg\Big)
\\ \nonumber+\int_G \Big[S\Big((id_d\otimes tr_n)(\rho_g)\Big)-S\Big(\big(id_d\otimes \Ne_j\big)(\rho_g)\Big)\Big]dg \Big\}.
\end{align*}
Now, for every $j$ we have that
\begin{align}\label{equation 1}
S\Big(\int_G \Ne_j\big((tr_d\otimes id_n)(\rho_g)\big)dg\Big)&=S\Big(\int_G \Ne_j\Big(\pi(g)^*\big((tr_d\otimes id_n)(\rho)\big)\pi(g)\Big)dg\Big)\\&\nonumber =S\Big(\int_G \sigma_j(g)^*\Ne_j\big((tr_d\otimes id_n)(\rho)\big)\sigma_j(g)dg\Big)\\&\nonumber =
S\Big(\frac{\uno_{n_j}}{n_j}\Big)=\ln n_j,
\end{align}where in the second equality we have used the covariant properties of our channel.

On the other hand, for every $j$ we also have
\begin{align}\label{equation 2}
\int_G S\Big((id_d\otimes tr_n)(\rho_g)\Big)dg&=\int_G S\Big((id_d\otimes tr_n)\big(\big(id_d\otimes \pi(g)^*\big)\rho \big(id_d\otimes \pi(g)\big)\big)\Big)dg\\&\nonumber=
\int_G S\big((id_d\otimes tr_n)(\rho)\big)dg=S\big((id_d\otimes tr_n)(\rho)\big),
\end{align}and
\begin{align}\label{equation 3}
S\Big(\big(id_d\otimes \Ne_j\big)(\rho_g)\Big)&=S\Big(\big(id_d\otimes \Ne_j\big)\Big(\big(id_d\otimes \pi(g)^*\big)\rho \big(id_d\otimes \pi(g)\big)\Big)\Big)\\&\nonumber=
S\Big(\big(\uno_d\otimes \sigma_j(g)^*\big)\big(id_d\otimes \Ne_j\big)(\rho)\big(\uno_d\otimes  \sigma_j(g)\big)\Big)\\&\nonumber=S\Big(\big(id_d\otimes \Ne_j\big)(\rho)\Big).
\end{align}Here in the last equality we have used that the von Neumann entropy is invariant under unitaries. Equations (\ref{equation 1}), (\ref{equation 2}) and (\ref{equation 3})  imply that
\begin{align*}
C^d_{prod}(\Ne)\geq \sum_{j=1}^m\mu_j\ln n_j+ \sum_{j=1}^m\mu_j\Big[S\Big((id_d\otimes tr_n)(\rho)\Big)-S\Big(\big(id_d\otimes \Ne_j\big)(\rho)\Big)\Big].
\end{align*}Since this happens for every state $\rho\in S_1^d\otimes S_1^n$, we conclude that
\begin{align*}
C^d_{prod}(\Ne)\geq  \sum_{j=1}^m\mu_j \ln n_j+ V_d(\Ne).
\end{align*}
\end{proof}
\section{$d$-restricted capacity of the quantum depolarizing channel}\label{Sec: Depolarizing channel}
In this section we will prove the part of Theorem \ref{Main theorem Depolarizing} corresponding to the depolarizing channel (Equation (\ref{p-norm depol})) and also Corollary \ref{Corollary Depolarizing}. Finally, we will see how to obtain the first part of Theorem \ref{Main Theorem Erasure} (Equation (\ref{sharp bounds dep})) by assuming (\ref{Eq multiplicativity erasure}), which will be proved in the next section.

It is very easy to see that $\mathcal \mathcal D_\lambda$ is a covariant channel with respect to $(\mathbb U(n), id_{\mathbb U(n)}, id_{\mathbb U(n)})$.
\begin{comment}
Indeed, it is very easy to see that Property 1 in Definition \ref{Def:covariant} is verified by this choice. Moreover, Property 2 in Definition \ref{Def:covariant} follows from the fact that for every unitary $U\in \mathbb U(n)$ we have
\begin{align*}
\mathcal \mathcal D_\lambda(U^*\rho U)=\lambda U^*\rho U+ (1-\lambda)\frac{1}{n}tr(U^*\rho U)\uno_n=U^*\Big(\lambda \rho+ (1-\lambda)\frac{1}{n}tr(\rho)\uno_n\Big)U.
\end{align*}
\end{comment}
%
Therefore, according to Proposition \ref{prop:C^d_{prod} for covariant direct sume} and Lemma \ref{lemma:V_d vs S_d}, the expression for $C_{prod}^d(\mathcal \mathcal D_\lambda)$ in Theorem \ref{Main theorem Depolarizing} can be obtained from Equation (\ref{p-norm depol}) by differentiation (and adding a $\ln n$ term). Indeed, if we differentiate in Equation (\ref{p-norm depol}) we obtain
$$\begin{array}{l}
\frac{d}{dp}\big\|\mathcal \mathcal \mathcal D_\lambda:S_1^n\rightarrow S_p^n\big\|_d|_{p=1}=
%\frac{d}{dp}\Big(\frac{1}{d}\Big(\lambda d+ \frac{1-\lambda}{n}\Big)^p+\big(\frac{1-\lambda}{n}\big)^p(n-\frac{1}{d})\Big)^\frac{1}{p}|_{p=1}
%=\frac{1}{d}\Big[(\lambda d+ \frac{1-\lambda}{n})\ln (\lambda d+ \frac{1-\lambda}{n})\Big]+ (n-\frac{1}{d})(\frac{1-\lambda}{n})\ln  (\frac{1-\lambda}{n})\\   \\=
%(\lambda +\frac{1-\lambda}{nd})\ln \big(\lambda +\frac{1-\lambda}{nd}\big)+ \big(\lambda +\frac{1-\lambda}{nd}\big)\ln d+ (nd-1)(\frac{1-\lambda}{nd})\ln (\frac{1-\lambda}{nd})+ (nd-1)(\frac{1-\lambda}{nd})\ln d\\    \\=
(\lambda +\frac{1-\lambda}{nd})\ln (\lambda +\frac{1-\lambda}{nd})+(nd-1)(\frac{1-\lambda}{nd})\ln (\frac{1-\lambda}{nd})+\ln d.
\end{array}$$Adding a $\ln n$ term we obtain desired equation\footnote{Recall that, according to Remark \ref{log 2 vs lgn}, we must replace our $\ln$-terms by $\log_2$-terms in order to consider the right capacity.}.

In order to prove (\ref{p-norm depol}) we will start by defining the following family of linear maps\footnote{It is very easy to see that $\theta_\lambda^{d,1}(\rho)$ is a quantum channel. However, we will consider the whole family $\big(\theta_\lambda^{d,p}(\rho)\big)_p$ in order to compute the $(1,p)$-norm of our channel.}:
$\theta_\lambda^{d,p}:S_1^d\rightarrow S_1^d\oplus_1S_1^d\subseteq S_1^{2d}$ for every $p\geq 1$, define by
\begin{align}\label{mod channel dep}
\theta_\lambda^{d,p}(\rho)=\Big(\lambda\rho+ \frac{1-\lambda}{n}tr(\rho)\uno_d\Big)\oplus  \frac{1-\lambda}{n}tr(\rho)\big(\frac{n-d}{d}\big)^\frac{1}{p}\uno_d
\end{align}for every $\rho\in S_1^d$.
\begin{prop}\label{prop upper bound by new channel}
Let $\mathcal \mathcal D_\lambda:S_1^n\rightarrow S_1^n$ be the quantum depolarizing channel with parameter $\lambda$ and $\theta_\lambda^{d,p}$ defined as above. Then,
\begin{align*}
\big\|\mathcal \mathcal D_\lambda:S_1^n\rightarrow S_p^n\big\|_d\leq \big\|\theta_\lambda^{d,p}:S_1^d\rightarrow S_p^d\oplus_pS_p^d\big\|_{cb}.
\end{align*}
\end{prop}
Before proving the proposition, we will show the following easy lemma.
\begin{lemma}\label{lemma V}
Given $1\leq d\leq n$, let us define the linear map $V:S_p^{d}\rightarrow S_p^{n-d}$ by
\begin{align*}
V(\rho)=\frac{tr(\rho)}{(n-d)^\frac{1}{p}d^\frac{1}{p'}}\uno_{n-d}, \text{     }\text{      }\rho\in S_p^d
\end{align*}Then, $\|V\|_{cb}=1$. Moreover,
\begin{align*}
\big\|id\oplus V:S_p\oplus_p S_p^{d} \rightarrow S_p\oplus_p S_p^{n-d}\big\|_{cb}=1.
\end{align*}
\end{lemma}
\begin{proof}
Since $V$ has rank one, we know that $\|V\|_{cb}=\|V\|$. Let us then consider an element $\rho$ in the unit ball of $S_p^{d}$. We have that
\begin{align*}
\|V(\rho)\|_{S_p^{n-d}}=\frac{|tr(\rho)|}{(n-d)^\frac{1}{p}d^\frac{1}{p'}}\|\uno_{n-d}\|_{S_p^{n-d}}\leq \frac{d^{\frac{1}{p'}}}{(n-d)^\frac{1}{p}d^\frac{1}{p'}}(n-d)^\frac{1}{p}=1.
\end{align*}The second statement follows straightforward from the first one.
\end{proof}
We prove now Proposition \ref{prop upper bound by new channel}.
\begin{proof}
According to (\ref{CB norm with general $t$}), it suffices to show that
\begin{align*}
\Big\|id_d\otimes \mathcal \mathcal D_\lambda:S_1^d(S_1^n)\rightarrow S_1^d(S_p^n)\Big\|\leq \Big \|id_d\otimes \theta_\lambda^{d,p}:S_1^d(S_1^d)\rightarrow S_1^d(S_p^d\oplus_pS_p^d)\Big\|.
\end{align*}In fact, since $\mathcal D_\lambda$ is completely positive we can restrict the computation of the first norm to positive elements (see Remark \ref{(p,1)}) so, by normalization, to states $\rho\in S_1^{dn}$. Moreover, since pure states are exactly the extreme points of general states, by convexity we can restrict to pure states $\xi=|\eta\rangle\langle\eta|\in S_1^{dn}$, where $\ba\eta\rangle$ is a unit vector in $\C^{dn}$. Now, according to the Hilbert-Schmidt decomposition we can assume that $\ba\eta\rangle=\sum_{i=1}^d\lambda_i |f_i\rangle\otimes |g_i\rangle$ for certain orthonormal systems $(|f_i\rangle)_i\subset \C^d$, $(|g_i\rangle)_i\subset \C^n$  respectively and  $\sum_{i=1}^d|\lambda_i|^2=1$. Moreover, by the unitary invariance of our channel $\mathcal \mathcal D_\lambda$ we can assume that $\ba\eta\rangle=\sum_{i=1}^d\lambda_ie_i\otimes e_i\in \C^d\otimes \C^d\subset \C^d\otimes \C^n$. Indeed, this is because we have
\begin{align*}
\big\|(id_d\otimes \mathcal \mathcal D_\lambda)(\xi)\big\|_{S_1^d(S_p^n)}&=\big\|(U\otimes V)\big((id_d\otimes \mathcal \mathcal D_\lambda)(\xi)\big)(U^*\otimes V^*)\big\|_{S_1^d(S_p^n)}\\&=
\big\|(id_d\otimes \mathcal \mathcal D_\lambda)\big((U\otimes V)\xi(U^*\otimes V^*)\big)\big\|_{S_1^d(S_p^n)}
\end{align*}for every $\xi$ and all unitaries $U\in M_d$ and $V\in M_n$. Therefore, $\xi=\sum_{i,j=1}^d\lambda_i\overline{\lambda_j}|i\rangle\langle j|\otimes |i\rangle\langle j|\in S_1^d\otimes S_1^n$. It is trivial to check that
\begin{align*}
(id_d\otimes \mathcal \mathcal D_\lambda)(\xi)=\lambda\xi+\frac{1-\lambda}{n}\sum_{i=1}^d|\lambda_i|^2|i\rangle\langle i|\otimes \uno_n.
\end{align*}Now, we can see that $\uno_n=\uno_d\oplus \uno_{n-d}$ and since $\xi\in S_1^d\otimes S_1^d\subset S_1^d\otimes S_1^n$, we have
\begin{align*}
(id_d\otimes \mathcal \mathcal D_\lambda)(\xi)=\Big(\lambda\xi+\frac{1-\lambda}{n}\sum_{i=1}^d|\lambda_i|^2|i\rangle\langle i|\otimes \uno_d\Big)\oplus \Big(\frac{1-\lambda}{n}\sum_{i=1}^d|\lambda_i|^2|i\rangle\langle i|\otimes \uno_{n-d}\Big).
\end{align*}
Let us now consider
\begin{align*}
(\uno_d\otimes \theta_\lambda^{d,p}(\xi))=\Big(\lambda\xi+\frac{1-\lambda}{n}\sum_{i=1}^d|\lambda_i|^2|i\rangle\langle i|\otimes \uno_d\Big)\oplus \Big(\frac{1-\lambda}{n}\big(\frac{n-d}{d}\big)^\frac{1}{p}\sum_{i=1}^d|\lambda_i|^2|i\rangle\langle i|\otimes \uno_d\Big).
\end{align*}
Since $\uno_{n-d}=V(\big(\frac{n-d}{d}\big)^\frac{1}{p} \uno_d)$, the result follows from Lemma \ref{lemma V}.
\end{proof}
In order to find an upper bound for the quantity $\big\|\theta_\lambda^{d,p}:S_1^d\rightarrow S_p^d\oplus_pS_p^d\big\|_{cb}$ we will use Theorem \ref{Theorem I: teleportation-general}. In the particular case we need, the theorem states that the map
\begin{align*}
j_p(\rho)=\frac{1}{d^\frac{1}{p}}\sum_{k,l=1}^d T_{k,l}\rho T_{k,l}^* \otimes e_{k,l}
\end{align*}defines a complete isometry of $S_p^d$ in $M_d(\ell_p^{d^2})$, which is complemented by a completely contractive and completely positive map. Moreover,
\begin{align*}
\tilde{J}_p(\rho_1\oplus \rho_2)=\frac{1}{d^\frac{1}{p}}\Big(\sum_{k,l=1}^d T_{k,l}\rho_1 T_{k,l}^* \otimes e_{k,l;1}\oplus \sum_{k,l=1}^d T_{k,l}\rho_2 T_{k,l}^*\otimes e_{k,l;2}\Big)
\end{align*}defines a complete isometry of $S_p^d\oplus_p S_p^d$ in $M_d(\ell_p^{d^2}\oplus_p \ell_p^{d^2})$ which is complemented by a completely contractive and completely positive map. Here, we denote by $e_{k,l;1}$ the elements of the canonical basis of the first $\ell_p^{d^2}$ space and by $e_{k,l;2}$ the elements of canonical basis of the second $\ell_p^{d^2}$ space.
\begin{lemma}\label{lemma norm commutative}
Let us consider the linear map $\Psi_{\alpha,\beta,\gamma}:\ell_1^{d^2}\rightarrow \ell_p^{d^2}\oplus_p \ell_p^{d^2}$ defined by
\begin{align*}
\Psi_{\alpha,\beta,\gamma}\Big(\sum_{i,j=1}^da_{i,j}e_{i,j}\Big)=\alpha\sum_{i,j=1}^da_{i,j}e_{i,j;1}+\beta\Big(\sum_{i,j=1}^da_{i,j}\Big)\Big(\sum_{i,j=1}^de_{i,j;1}\Big)+
\delta\Big(\sum_{i,j=1}^da_{i,j}\Big)\Big(\sum_{i,j=1}^de_{i,j;2}\Big).
\end{align*}Then,
\begin{align*}
\|\Psi_{\alpha,\beta,\gamma}\|_{cb}=\|\Psi_{\alpha,\beta,\gamma}\|=\Big(|\alpha+\beta|^p+(d^2-1)|\beta|^p+d^2|\delta|^p\Big)^\frac{1}{p}.
\end{align*}
\end{lemma}
\begin{proof}
The equality $\|\Psi_{\alpha,\beta,\gamma}\|_{cb}=\|\Psi_{\alpha,\beta,\gamma}\|$ follows from the fact that we consider the natural operator space structure on $\ell_1^{d^2}$, which is the maximal one (see \cite[Chapter 3]{Pisierbook}). On the other hand, in order to estimate $\|\Psi_{\alpha,\beta,\gamma}\|$ it suffices to check the elements of the canonical basis $e_{i,j}$. Moreover, by the symmetry of the problem is suffices to check $e_{1,1}$. Then,
\begin{align*}
\|\Psi_{\alpha,\beta,\gamma}\|=\|\Psi_{\alpha,\beta,\gamma}(e_{1,1})\|_{\ell_p^{d^2}\oplus_p \ell_p^{d^2}}&=\Big\|\alpha e_{1,1;1}+\beta\sum_{i,j=1}^de_{i,j;1}+ \gamma \sum_{i,j=1}^de_{i,j;2}\Big\|_{\ell_p^{d^2}\oplus_p \ell_p^{d^2}}\\&=\Big(|\alpha+\beta|^p+(d^2-1)|\beta|^p+d^2|\delta|^p\Big)^\frac{1}{p}.
\end{align*}
\end{proof}
The key result in our analysis is the following factorization.
\begin{prop}\label{prop factorization}
Let us fix $\alpha=\lambda d^\frac{1}{p'}$, $\beta=\frac{1-\lambda}{d^\frac{1}{p}n}$ and $\delta=\frac{1-\lambda}{d^\frac{1}{p}n}\big(\frac{n-d}{d}\big)^\frac{1}{p}$. Then, we have\
\begin{align*}
(id_d\otimes \Psi_{\alpha,\beta,\gamma})\circ j_1=\tilde{J}_p\circ \theta_\lambda^{d,p}.
\end{align*}
\end{prop}
\begin{proof}
Let consider an element $\rho\in S_1^d$. Then we have
$$\begin{array}{l}
\big((id_d\otimes \Psi_{\alpha,\beta,\gamma})\circ j_1\big)(\rho)=(id_d\otimes \Psi_{\alpha,\beta,\gamma})\big(\frac{1}{d}\sum_{k,l=1}^d T_{k,l}\rho T_{k,l}^* \otimes e_{k,l}\big)\\  \\=\frac{1}{d}\sum_{k,l=1}^d T_{k,l}\rho T_{k,l}^* \otimes \Psi_{\alpha,\beta,\gamma}(e_{k,l}) \\ \\=\frac{1}{d}\Big(\sum_{k,l=1}^d T_{k,l}\rho T_{k,l}^* \otimes \big(\alpha e_{k,l;1}+\beta\sum_{i,j=1}^de_{i,j;1}\oplus \gamma\sum_{i,j=1}^de_{i,j;2}\big)\Big) \\ \\=
\Big(\frac{\alpha}{d}\sum_{k,l=1}^d T_{k,l}\rho T_{k,l}^* \otimes e_{k,l;1}+ \beta tr(\rho) \uno_d\otimes \sum_{i,j=1}^de_{i,j;1}\Big)\oplus \Big(\gamma tr(\rho) \uno_d\otimes \sum_{i,j=1}^de_{i,j;2}\Big),
\end{array}$$where in the last step we have used that $\sum_{k,l=1}^d T_{k,l}\rho T_{k,l}^*= dtr(\rho) \uno_d$. Indeed, this can be easily checked by noting that
\begin{align*}
\sum_{k,l=1}^dT_{k,l}\ba p\rangle \langle q\ba T_{k,l}^*=\delta_{p,q}d\uno_d\text{        }\text{     for every    }\text{        } p,q=1,\cdots, d.
\end{align*}
If we consider the specific values for $\alpha$, $\beta$ and $\gamma$ stated in the proposition, we obtain
$$\Big(\frac{\lambda}{d^\frac{1}{p}}\sum_{k,l=1}^d T_{k,l}\rho T_{k,l}^* \otimes e_{k,l;1}+ \frac{1-\lambda}{d^\frac{1}{p}n} tr(\rho) \uno_d\otimes \sum_{i,j=1}^de_{i,j;1}\Big)\oplus \Big(\frac{1-\lambda}{d^\frac{1}{p}n}\big(\frac{n-d}{d}\big)^\frac{1}{p} tr(\rho) \uno_d\otimes \sum_{i,j=1}^de_{i,j;2}\Big).$$
On the other hand,
$$\begin{array}{l}
\big(\tilde{J}_p\circ \theta_\lambda^{d,p}(\rho)\big)=\tilde{J}_p\Big(\big(\lambda\rho+ \frac{1-\lambda}{n}tr(\rho)\uno_d\big)\oplus  \frac{1-\lambda}{n}tr(\rho)\big(\frac{n-d}{d}\big)^\frac{1}{p}\uno_d\Big)\\  \\=
\frac{1}{d^\frac{1}{p}}\Big(\sum_{k,l=1}^d T_{k,l}\big(\lambda\rho+ \frac{1-\lambda}{n}tr(\rho)\uno_d\big) T_{k,l}^* \otimes e_{k,l;1}\oplus \sum_{k,l=1}^d T_{k,l}\frac{1-\lambda}{n}tr(\rho)\big(\frac{n-d}{d}\big)^\frac{1}{p}\uno_d T_{k,l}^*\otimes e_{k,l;2}\Big),
\end{array}$$which is equal to
$$
\Big(\frac{\lambda}{d^\frac{1}{p}}\sum_{k,l=1}^d T_{k,l}\rho T_{k,l}^* \otimes e_{k,l;1}+ \frac{1-\lambda}{d^\frac{1}{p}n}tr(\rho) \uno_d\otimes \sum_{k,l=1}^d e_{k,l;1}\Big)\oplus\Big(\frac{1-\lambda}{d^\frac{1}{p}n}\Big(\frac{n-d}{d}\Big)^\frac{1}{p}tr(\rho)\uno_d\otimes  \sum_{k,l=1}^d e_{k,l;2}\Big).$$This concludes the proof.
\end{proof}
\begin{corollary}\label{corollary: prop upper bound}
Let $\theta_\lambda^{d,p}$ be the linear map defined in (\ref{mod channel dep}). Then,
\begin{align*}
\Big\|\theta_\lambda^{d,p}:S_1^d\rightarrow S_p^d\oplus_pS_p^d\Big\|_{cb}\leq \Big(\frac{1}{d}\Big(\lambda d+ \frac{1-\lambda}{n}\Big)^p+\Big(\frac{1-\lambda}{n}\Big)^p\Big(n-\frac{1}{d}\Big)\Big)^\frac{1}{p}.
\end{align*}
\end{corollary}
\begin{proof}
By Proposition \ref{prop factorization} and the fact that $j_p$ and $\tilde{J}_p$ are complete isometries it suffices to show that
\begin{align*}
\Big\|(id_d\otimes \Psi_{\alpha,\beta,\gamma}):M_d(\ell_1^{d^2})\rightarrow M_d(\ell_p^{d^2}\oplus_p \ell_p^{d^2})\Big\|_{cb}\leq \Big(\frac{1}{d}\Big(\lambda d+ \frac{1-\lambda}{n}\Big)^p+\Big(\frac{1-\lambda}{n}\Big)^p\Big(n-\frac{1}{d}\Big)\Big)^\frac{1}{p}.
\end{align*}
Now, it follows from the definition of the completely bounded norm that
\begin{align*}
\big\|(id_d\otimes \Psi_{\alpha,\beta,\gamma}:M_d(\ell_1^{d^2})\rightarrow M_d(\ell_p^{d^2}\oplus_p \ell_p^{d^2})\big\|_{cb}=\big\|\Psi_{\alpha,\beta,\gamma}:\ell_1^{d^2}\rightarrow \ell_p^{d^2}\oplus_p \ell_p^{d^2}\big\|_{cb}\\= \Big(|\alpha+\beta|^p+(d^2-1)|\beta|^p+d^2|\delta|^p\Big)^\frac{1}{p},
\end{align*}where the last equality follows from Lemma \ref{lemma norm commutative}.
By considering the values  for $\alpha$, $\beta$ and $\gamma$ stated in Proposition \ref{prop factorization}, we obtain
\begin{align*}
\big\|(id_d\otimes \Psi_{\alpha,\beta,\gamma})\big\|_{cb}^p&=\Big(\lambda d^\frac{1}{p'}+ \frac{1-\lambda}{d^\frac{1}{p}n}\Big)^p+ (d^2-1)\Big(\frac{1-\lambda}{d^\frac{1}{p}n}\Big)^p+ d^2\Big(\frac{1-\lambda}{d^\frac{1}{p}n}\Big(\frac{n-d}{d}\Big)^\frac{1}{p}\Big)^p\\&=
\frac{1}{d}\Big(\lambda d+ \frac{1-\lambda}{n}\Big)^p+\frac{d(1-\lambda)^p}{n^p}-\frac{(1-\lambda)^p}{dn^p}+\frac{(1-\lambda)^p}{n^{p-1}}-\frac{d(1-\lambda)^p}{n^p}\\&=
\frac{1}{d}\Big(\lambda d+ \frac{1-\lambda}{n}\Big)^p+\big(\frac{1-\lambda}{n}\big)^p(n-\frac{1}{d}).
\end{align*}
\end{proof}
We are now ready to prove (\ref{p-norm depol}).
\begin{proof}[Proof of Equation (\ref{p-norm depol}) in Theorem \ref{Main theorem Depolarizing}]
The upper bound in Equation (\ref{p-norm depol}) follows from Proposition \ref{prop upper bound by new channel} and Corollary \ref{corollary: prop upper bound}. Thus, we must only show the lower bound.

Let us consider the particular element $\xi=\frac{1}{d}\sum_{i,j=1}^d|i\rangle \langle j|\otimes |i\rangle \langle j|\in M_d(M_n)$. We have already mentioned that for a positive element $\xi$ in $M_d(M_n)$ one has
\begin{align*}
\|\xi\|_{S_p^d(S_1^n)}=\big\|(id_d\otimes tr_n)(\xi)\big\|_{S_p^d}=\frac{1}{d}\|\uno_d\|_{S_p^d}=\frac{d^\frac{1}{p}}{d}=\frac{1}{d^\frac{1}{p'}}.
\end{align*}On the other hand,
\begin{align*}
\big\|(id_d\otimes \mathcal \mathcal D_\lambda)(\xi)\big\|_{S_p^d(S_p^n)}=\big\|\lambda\xi +(1-\lambda)\frac{\uno_{nd}}{nd}\big\|_{S_p^{dn}}.
\end{align*}Then, using that $\xi=|\eta\rangle\langle \eta|$ is a pure state with $\eta=\frac{1}{\sqrt{d}}\sum_{i=1}^d\ba ii\rangle $, the element $\lambda\xi +(1-\lambda)\frac{\uno_{nd}}{n}$ can be seen as a matrix in $M_{nd}$ with all eigenvalues equal $\frac{1-\lambda}{nd}$ up to one which is $\lambda+\frac{1-\lambda}{nd}$. Hence,
\begin{align*}
\big\|(id_d\otimes \mathcal \mathcal D_\lambda)(\xi)\big\|_{S_p^d(S_p^n)}=\Big(\Big(\lambda+\frac{1-\lambda}{nd}\Big)^p+ (nd-1)\Big(\frac{1-\lambda}{nd}\Big)^p\Big)^\frac{1}{p}.
\end{align*}We immediately conclude that
\begin{align*}
\big\|id_d\otimes \mathcal \mathcal D_\lambda:S_p^d(S_1^n)\rightarrow S_p^d(S_p^n)\big \|_d\geq d^\frac{1}{p'}\Big(\big(\lambda+\frac{1-\lambda}{nd}\big)^p+ (nd-1)\Big(\frac{1-\lambda}{nd}\Big)^p\Big)^\frac{1}{p}.
\end{align*}
Now, it is very easy to see that this is exactly the same expression as the one in Equation (\ref{p-norm depol}). Indeed,
\begin{align*}
&d^\frac{1}{p'}\Big(\big(\lambda+\frac{1-\lambda}{nd}\big)^p+ (nd-1)\big(\frac{1-\lambda}{nd}\big)^p\Big)^\frac{1}{p}\\&=\Big(d^{p-1}\Big[\frac{1}{d^p}\Big(\lambda d+\frac{1-\lambda}{n}\Big)^p+ \frac{1}{d^p}(nd-1)\Big(\frac{1-\lambda}{n}\Big)^p\Big]\Big)^\frac{1}{p}\\&=\Big(\frac{1}{d}\Big[\Big(\lambda d+\frac{1-\lambda}{n}\Big)^p+(nd-1)\Big(\frac{1-\lambda}{n}\Big)^p\Big]\Big)^\frac{1}{p}\\&=\Big(\frac{1}{d}\Big(\lambda d+\frac{1-\lambda}{n}\Big)^p+(n-\frac{1}{d})\Big(\frac{1-\lambda}{n}\Big)^p\Big)^\frac{1}{p}.
\end{align*}
Therefore, the result follows.
\end{proof}
\subsection{Non additivity of $C^d_{prod}$ for the depolarizing channel}
As we said in the previous section the quantity $C_{prod}^d(\mathcal \mathcal D_\lambda)$ in Theorem \ref{Main theorem Depolarizing} extends the corresponding results for the product state classical capacity of the quantum depolarizing channel (with no assisted entanglement), so $d=1$, and for the product state (unlimited) assisted entanglement classical capacity, $d=n$. In fact, it is known that in both cases the quantity $C_{prod}^d(\mathcal \mathcal D_\lambda)$ coincides with the capacity $C^d(\mathcal \mathcal D_\lambda)$. Somehow surprisingly, this is no longer true if $1<d< n$ as we stated in Corollary \ref{Corollary Depolarizing}.

First of all, note that it is very easy to see that
\begin{align}\label{trivial lower bound for the tensor}
C_{prod}^{d^2}(\mathcal \mathcal D_\lambda\otimes \mathcal \mathcal D_\lambda)\geq C_{prod}^{d^2}(\mathcal \mathcal D_\lambda)+C_{prod}^1(\mathcal \mathcal D_\lambda).
\end{align}
Indeed, from a physical point of view this means that a particular strategy for Alice and Bob with a $d^2$-dimensional entangled state consists of using all the entanglement in one of the channel and using the other channel without assisted entanglement. From a mathematical point of view, this can be deduced from the fact that
\begin{align*}
\big\|\mathcal \mathcal D_\lambda\otimes \mathcal \mathcal D_\lambda:S_1^n\otimes_1 S_1^n\rightarrow S_p^n\otimes_p S_p^n\big\|_{d^2}\geq
\big\|\mathcal \mathcal D_\lambda:S_1^n\rightarrow S_p^n\big\|_{d^2}+
\big\|\mathcal \mathcal D_\lambda:S_1^n\rightarrow S_p^n\big\|,
\end{align*}which is obvious by restricting to elements of the form $x=y\otimes z$, with $y\in S_p^{d^2}(S_1^n)$ and $z\in S_1^n$ in the computation of the norm. The fact that we have a complete description of $C_{prod}^d(\mathcal \mathcal D_\lambda)$ for every $n$, $d$ and $\lambda$ allows us to exactly compute the quantity
\begin{align}
f(n,d,\lambda)=C_{prod}^{d^2}(\mathcal \mathcal D_\lambda)+C_{prod}^1(\mathcal \mathcal D_\lambda)-2C_{prod}^{d}(\mathcal \mathcal D_\lambda).
\end{align}According to (\ref{trivial lower bound for the tensor}), we want to show that $f(n,d,\lambda)$ is strictly positive for some values of $n$, $d$ and $\lambda$. Now,
\begin{align*}
f(n,d,\lambda)&=\big(\lambda+\frac{1-\lambda}{nd^2}\big)\ln \big(\lambda+\frac{1-\lambda}{nd^2}\big)+(nd^2-1)\big(\frac{1-\lambda}{nd^2}\big)\ln \big(\frac{1-\lambda}{nd^2}\big)
\\&+\big(\lambda+\frac{1-\lambda}{n}\big)\ln \big (\lambda+\frac{1-\lambda}{n}\big)+(n-1)\big(\frac{1-\lambda}{n}\big)\ln \big(\frac{1-\lambda}{n}\big)
\\&-2\big(\lambda+\frac{1-\lambda}{nd}\big)\ln \big(\lambda+\frac{1-\lambda}{nd}\big)-2(nd-1)\big(\frac{1-\lambda}{nd}\big)\ln \big(\frac{1-\lambda}{nd}\big).
\end{align*}The most basic example\footnote{It can be shown that for $n=3$, $C_{prod}^3(\mathcal \mathcal D_\lambda)+C_{prod}^1(\mathcal \mathcal D_\lambda)-2C_{prod}^2(\mathcal \mathcal D_\lambda)<0$ for every $\lambda\in (0,1)$.} can be found for $n=4$ and $d=2$. The function $h(\lambda)=f(4,2,\lambda)$ is represented below.
\begin{figure*}[h]
\begin{center}
   %Requires \usepackage{graphicx}
  \includegraphics[width=5cm]{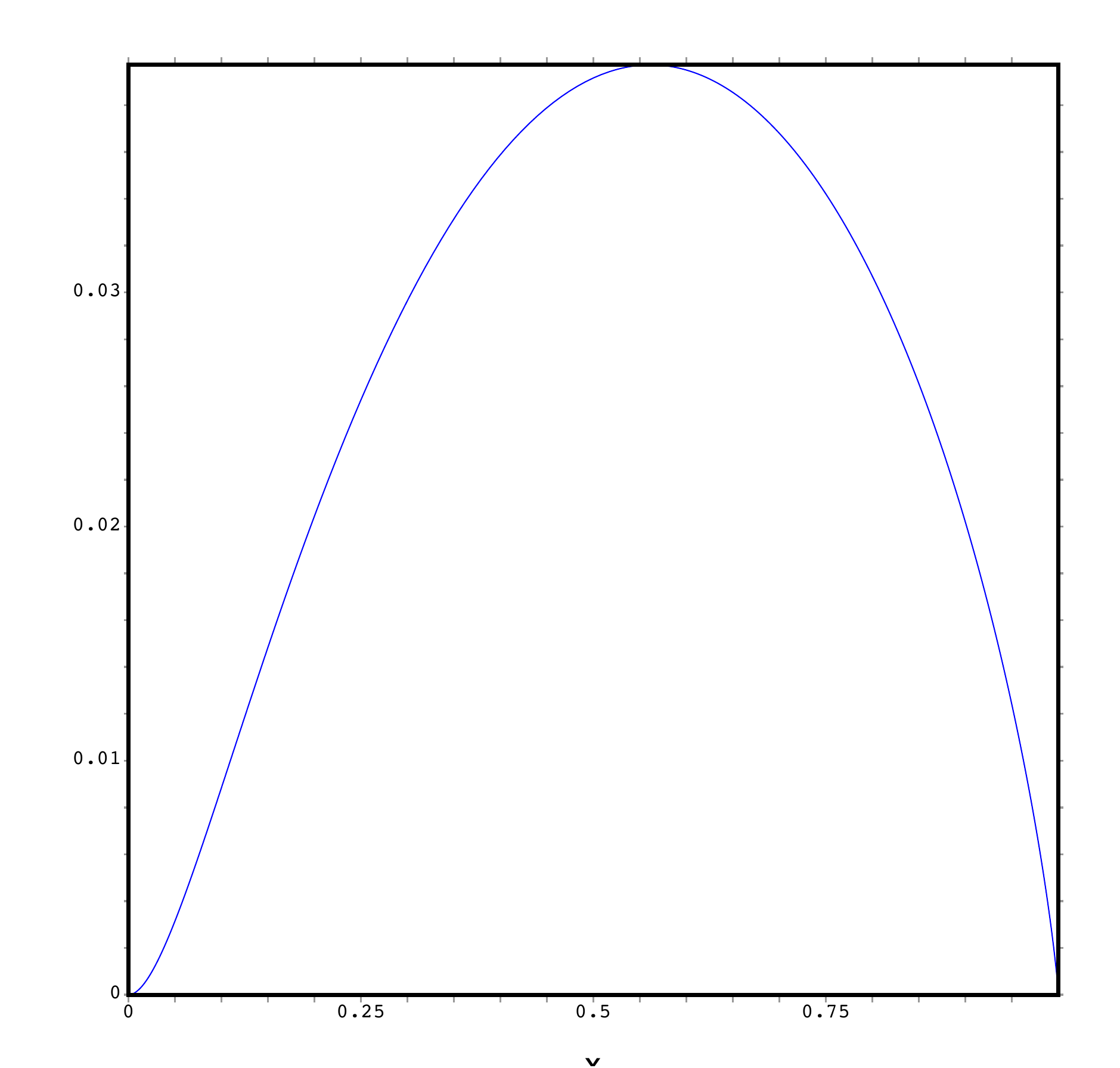}\\
  %\caption{Geometric meaning of $\omega_\Psi(S)$.}
\end{center}
\end{figure*}
Recall that, according to Remark \ref{log 2 vs lgn}, in order to compute the real quantity $C^d_{prod}(\mathcal \mathcal D_\lambda)$ we must multiply by $\frac{1}{\ln 2}$. We can see that the ``amount of violation'' $h(\lambda)$ is very small. Some other examples can be found where the amount of violation is arbitrary large. Indeed, it was shown in \cite[Theorem 1.2]{JuPa2} that for every natural number $n$, one can find a quantum channel $\Ne:S_1^{2n}\rightarrow S_1^{2n}$ such that $$C_{prod}^{n}(\Ne\otimes \Ne)- 2C_{prod}^{\sqrt{n}}(\Ne)\succeq \frac{1}{3}\log_2 n,$$where we use the symbol $\succeq$ to denote inequality up to universal (additive) constants which do not depend on $n$. One could wonder whether we can have a similar result for the quantum depolarizing channel so that the reason for our small value in the violation is that we are considering parameters $n$ and $d$ very small. In fact, our Theorem \ref{Main Theorem Erasure} (Equation (\ref{sharp bounds dep})) s!
 hows that
  for the quantum depolarizing channel the amount of violation is bounded by $\ln 2$ independently of $n$ and $d$ (and the number of uses of the channel). To finish this section we will prove (\ref{sharp bounds dep}) by assuming Equation (\ref{Eq multiplicativity erasure}), which will be proved in the next section.
\begin{proof}[Proof of Equation (\ref{sharp bounds dep}) in Theorem \ref{Main Theorem Erasure}]
Equation (\ref{Eq multiplicativity erasure}) states that $C^d(\mathcal E_\lambda)=\lambda \ln (nd)$, where $\mathcal E_\lambda:S_1^n\rightarrow S_1^n\oplus_1 \C$ denotes the quantum erasure channel with parameter $\lambda$, defined by $$\mathcal E_\lambda(\rho)=\lambda \rho \oplus (1-\lambda)tr(\rho)   \text{       }\text {       for every     }\text{       }  \rho\in S_1^n.$$ Since it is very easy to see that $C^d(\mathcal D_\lambda)\leq C^d(\mathcal E_\lambda)$, the last inequality in (\ref{sharp bounds dep}) follows. On the other hand, we know that the inequality $C_{prod}^d(\mathcal D_\lambda)\leq C^d(\mathcal D_\lambda)$ holds for every channel. Therefore, we just need to show the first inequality in (\ref{sharp bounds dep}). To this end, note that
\begin{align*}
C_{prod}^d(\mathcal D_\lambda)&= \ln (nd)+ \mu\ln \mu+ \Big(\frac{nd-1}{nd}\Big)(1-\lambda)\ln \Big(\frac{1-\lambda}{nd}\Big)\\&=
\ln (nd)+ \mu\ln \mu+\Big(\frac{nd-1}{nd}\Big)(1-\lambda)\Big[\ln \Big(\frac{1-\lambda}{nd}\Big)+\ln(nd-1)-\ln(nd-1)\Big]\\&=
\ln (nd)-H(\mu,1-\mu)-\Big(\frac{nd-1}{nd}\Big)(1-\lambda)\ln(nd-1)\\&=
\Big(1- \Big(\frac{nd-1}{nd})(1-\lambda\Big)\Big)\ln(nd)- H(\mu,1-\mu)- (\frac{nd-1}{nd})(1-\lambda)\ln(\frac{nd-1}{nd})\\&=
\big(\lambda+ \frac{1-\lambda}{nd}\big)\ln(nd)- H(\mu,1-\mu)- \Big(\frac{nd-1}{nd}\Big)(1-\lambda)\ln \Big(\frac{nd-1}{nd}\Big)\\& \geq \lambda\ln (nd)-H(\mu,1-\mu).
\end{align*}
\end{proof}
\section{$d$-restricted capacity of the quantum erasure channel}\label{Sec: Erasure channel}
In this section we will prove the part of Theorem \ref{Main theorem Depolarizing} and Theorem \ref{Main Theorem Erasure} corresponding to the quantum erasure channel. We will start showing Equation (\ref{p-norm erasure}) in Theorem \ref{Main theorem Depolarizing}. As in the case of the quantum depolarizing channel, it is very easy to see that the quantum erasure channel is covariant. In fact, one can also easily check that the channel $\mathcal E_\lambda^{\otimes_k}$ is covariant for every $k$, according to our Definition \ref{Def:covariant}. Let us show the case $k=2$ as an illustration. In this case the channel $$\mathcal E_\lambda^{\otimes_2}:S_1^{n^2}\rightarrow S_1^{n^2}\oplus_1 S_1^n \oplus_1 S_1^n \oplus \C$$is given by $$\mathcal E_\lambda^{\otimes_2}(\rho)=\lambda^2 \rho\oplus \lambda(1-\lambda)(id_n\otimes tr_n)(\rho)\oplus \lambda(1-\lambda)(tr_n\otimes id_n)(\rho)\oplus(1-\lambda)^2(tr_n\otimes tr_n)(\rho).$$Then, we can consider the group $G=\mathbb U(n)\times \!
 mathbb U(
 n)$ together with the representations $\pi=\sigma_1=\uno_{\mathbb U(n)\times \mathbb U(n)}$, $\sigma_2=\Pi_1$, $\sigma_3=\Pi_2$ and $\sigma_4=\uno_{1,1}\circ\Pi_1$, where $\Pi_1:\mathbb U(n)\times \mathbb U(n)\rightarrow \mathbb U(n)$ is the projection onto the first copy, $\Pi_2$ is the projection onto the second copy and $\uno_{1,1}:\mathbb U(n)\rightarrow \mathbb U(1)$ is the $1$-dimensional unitary representation, given by $U\mapsto \langle 1\ba U\ba 1\rangle$. Then, one can see that Properties 1 and 2 in Definition \ref{Def:covariant} are verified by this choice.

Note that, according to Proposition \ref{prop:C^d_{prod} for covariant direct sume} we have that
\begin{align*}
C^d_{prod}(\mathcal E_\lambda)&=\lambda \ln  n+ V_d(\mathcal E_\lambda)\\&=\lambda \ln  n+ \sup \Big\{S\Big((id_d\otimes tr_n)(|\psi\rangle\langle \psi|)\Big)-\lambda S\Big((id_d\otimes id_n)(|\psi\rangle\langle \psi|)\Big)\\&- (1-\lambda)S\Big((id_d\otimes tr_n)(|\psi\rangle\langle \psi|)\Big)\Big\}\\&=\lambda \ln  n+ \lambda \sup \Big\{S\Big((id_d\otimes tr_n)(|\psi\rangle\langle \psi|)\Big)\Big\}\leq \lambda \ln  n+ \lambda \ln  d=\lambda \ln  (nd).
\end{align*}Here, the supremum runs over all pure states $|\psi\rangle\in \C^d\otimes \C^n$ and we have used that $S(\rho)=0$ for every pure state $\rho$ and also that $S(\eta)\leq \ln d$ for every $d$-dimensional state $\eta$.

On the other hand, if we consider the $d$-maximally entangled state $|\psi_d\rangle=\frac{1}{\sqrt{d}}\sum_{i=1}^d|e_i\otimes e_i\in \C^d\otimes \C^n$ we can check that
\begin{align*}
C^d_{prod}(\mathcal E_\lambda)\geq \lambda \ln  n+ \lambda S\Big((id_d\otimes tr_n)(|\psi_d\rangle\langle \psi_d|)\Big)= \lambda \ln  n+ \lambda \ln  d=\lambda \ln  (nd).
\end{align*}Therefore, the previous argument already gives us the right expression for $C^d_{prod}(\mathcal E_\lambda)$. However, in this work we are interested in computing the $d$-norms of the channels, so we will show here Equation (\ref{p-norm erasure}) from which the previous quantity can be obtained by differentiating (and adding an extra $\ln$-term). It is interesting to remark here that, computing the $d$-norm of a channel is a stronger result than computing its capacity. This point will be particularly important in the study of $\mathcal E_\lambda^k$ below, since we couldn't find a good expression for its $d^k$-norm and we directly computed $C^{d^k}_{prod}(\mathcal E_\lambda^{\otimes_k})$.
\begin{proof}[Proof of Equation (\ref{p-norm erasure}) in Theorem \ref{Main theorem Depolarizing}]
Let us first note that
\begin{align*}
\Big \|\mathcal E_\lambda:S_1^n\rightarrow S_p^n\oplus_p \C\Big \|_d=\Big \|id_d\otimes \mathcal E_\lambda:S_p^d(S_1^n)\rightarrow S_p^d(S_p^n)\oplus_p S_p^d\Big \|.
\end{align*}In order to compute this norm, let us consider an element $\rho\in M_d\otimes M_n$ with $\|\rho\|_{S_p^d(S_1^n)}=1$. It is very easy that this implies, in particular, that $\|(id_d\otimes tr_n)(\rho)\|_{S_p^d}\leq 1$. Indeed, this is
a trivial consequence of the fact that $tr:S_1^n\rightarrow \C$ is a (complete) contraction. On the other hand,
\begin{align*}
(id_d\otimes \mathcal E_\lambda)(\rho)=\lambda \rho \oplus (1-\lambda)(id_d\otimes tr_n)(\rho).
\end{align*}Thus,
\begin{align*}
\big\|(id_d\otimes \mathcal E_\lambda)(\rho)\big\|_{S_p^d(S_p^n)\oplus_p S_p^d}&=\Big(\lambda^p\|\rho\|_{S_p^{dn}}^p+(1-\lambda)^p\big\|(id_d\otimes tr_n)(\rho)\big\|_{S_p^d}^p\Big)^\frac{1}{p}\\&
\leq\Big(\lambda^p\|\rho\|_{S_p^{dn}}^p+(1-\lambda)^p\Big)^\frac{1}{p}\\&\leq \Big(\lambda^pd^{p-1}+(1-\lambda)^p\Big)^\frac{1}{p}.
\end{align*}Here, in the last inequality we have used that $$\big\|id_n: S_1^n \rightarrow S_p^n\big\|_d=\big\|id_d\otimes id_n: S_p^d(S_1^n) \rightarrow S_p^d(S_p^n)\big\|=d^{1-\frac{1}{p}}.$$
On the other hand, one can see that
\begin{align*}
\big\|id_d\otimes \mathcal E_\lambda:S_p^d(S_1^n)\rightarrow S_p^d(S_p^n)\oplus_p S_p^d\big\|\geq \Big(\lambda^pd^{p-1}+(1-\lambda)^p\Big)^\frac{1}{p},
\end{align*}by testing this norm at the $d$-maximally entangled state $\rho=|\psi_d\rangle\langle\psi_d|$.
\end{proof}
In order to show Equation (\ref{Eq multiplicativity erasure}) in Theorem \ref{Main Theorem Erasure} we must deal with an arbitrary number of tensor products of the channel $\mathcal E_\lambda$. To this end, we need to introduce some notation. Let us fix $k\in \N$ and consider a natural number $s$ with $0\leq s\leq k$. We note that there are $\binom{k}{s}$ subsets $A$ of $\{1,\cdots, k\}$ with cardinal $|A|=s$. For each of these sets we will denote $$\Ne_A:S_1^{n^k}\rightarrow S_1^{n^{|A|}},$$ defined by $$\Ne_A(\rho)=(id_A\otimes tr_{A^c})(\rho)\text{      }\text{   for every    }\text{        }  \rho\in S_1^{n^k},$$where $(id_A\otimes tr_{A^c})(\rho)\in  S_1^{n^{|A|}}$ denotes the state $\rho$ after tracing out all the systems $j\in A^c$. Then, it is clear that
\begin{align*}
\mathcal E_\lambda^{\otimes_k}:S_1^{n^k}\rightarrow \bigoplus_{s=0}^k\bigoplus_{\substack{A\subseteq \{1,\cdots, k\}\\|A|=s}}S_1^{n^{|A|}}
\end{align*}is given by
\begin{align*}
\mathcal E_\lambda^{\otimes_k}(\rho)= \bigoplus_{s=0}^k\bigoplus_{\substack{A\subseteq \{1,\cdots, k\}\\|A|=s}}\lambda^s(1-\lambda)^{k-s}\Ne_A(\rho) \text{      }\text{   for every    }\text{        }  \rho\in S_1^{n^k}.
\end{align*}
\begin{lemma}\label{lemma:Upper bound C^d_{prod} erasure}
For every $k\in \N$ we have
\begin{align*}
C^d_{prod}(\mathcal E_\lambda^{\otimes_k})\leq \sum_{s=1}^k\binom{k}{s}\lambda^s(1-\lambda)^{k-s}\ln  n^s+\sum_{s=0}^k\binom{k}{s}\lambda^s(1-\lambda)^{k-s}V_d(\Ne_s).
\end{align*}Here, $$\Ne_s:S_1^{n^k}\rightarrow \bigoplus_{\substack{A\subseteq \{1,\cdots, k\}\\|A|=s}}S_1^{n^{|A|}}$$is defined by
\begin{align*}
\Ne_s(\rho)=\frac{1}{\binom{k}{s}}\bigoplus_{\substack{A\subseteq \{1,\cdots, k\}\\|A|=s}}\Ne_A(\rho)\text{      }\text{   for every    }\text{        }  \rho\in S_1^{n^k}.
\end{align*}
\end{lemma}
\begin{proof}
According to Proposition \ref{prop:C^d_{prod} for covariant direct sume} and the covariant property of $\mathcal E_\lambda^{\otimes_k}$ we have that
\begin{align*}
C^d_{prod}(\mathcal E_\lambda^{\otimes_k})= \sum_{s=1}^k\binom{k}{s}\lambda^s(1-\lambda)^{k-s}\ln  n^s+ V_d(\mathcal E_\lambda^{\otimes_k}).
\end{align*}On the other hand, by definition, $V_d(\mathcal E_\lambda^{\otimes_k})$ is equal to
\begin{align*}
&\sup\Big\{S\Big((id_d\otimes tr_{n^k})(|\psi\rangle\langle \psi|)\Big)-\sum_{s=0}^k\sum_{\substack{A\subseteq \{1,\cdots, k\}\\|A|=s}}\lambda^s(1-\lambda)^{k-s}S\Big((id_d\otimes \Ne_A)(|\psi\rangle\langle \psi|)\Big)\Big\}=\\&
\sup\Big\{\sum_{s=0}^k\binom{k}{s}\lambda^s(1-\lambda)^{k-s}\Big[S\Big((id_d\otimes tr_{n^k})(|\psi\rangle\langle \psi|)\Big)-\frac{1}{\binom{k}{s}}\sum_{\substack{A\subseteq \{1,\cdots, k\}\\|A|=s}}S\Big((id_d\otimes \Ne_A)(|\psi\rangle\langle \psi|)\Big)\Big]\Big\},
\end{align*}where the supremum is taking over all pure states $\ba\psi\rangle\in \C^d\otimes \C^{n^k}$. Here, we have used the identity
\begin{align}\label{identity}
1=\big(\lambda+ (1-\lambda)\big)^k=\sum_{s=0}^k\binom{k}{s}\lambda^s(1-\lambda)^{k-s}.
\end{align}
It follows now easily that the previous quantity is lower than or equal to
\begin{align*}
&\sum_{s=0}^k\binom{k}{s}\lambda^s(1-\lambda)^{k-s}\sup\Big\{S\Big((id_d\otimes tr_{n^k})(|\psi\rangle\langle \psi|)\Big)-\frac{1}{\binom{k}{s}}\sum_{\substack{A\subseteq \{1,\cdots, k\}\\|A|=s}}S\Big((id_d\otimes \Ne_A)(|\psi\rangle\langle \psi|)\Big)\Big\}\\=&\sum_{s=0}^k\binom{k}{s}\lambda^s(1-\lambda)^{k-s}V_d(\Ne_s),
\end{align*}where all the supremums are taking over all pure states $\ba\psi\rangle\in \C^d\otimes \C^{n^k}$.

The statement of the lemma follows.
\end{proof}
\begin{lemma}\label{lemma add erasure}
Let $k$ and $s$ be two natural numbers such that $0\leq s\leq k$. Then,
\begin{align*}
V_d(\Ne_s)\leq \frac{s}{k}\ln  d.
\end{align*}
\end{lemma}
Before proving this lemma, we will show how to deduce the main result of this section from Lemma \ref{lemma:Upper bound C^d_{prod} erasure} and Lemma \ref{lemma add erasure}.
\begin{proof}[Proof of Equation (\ref{Eq multiplicativity erasure}) in Theorem \ref{Main Theorem Erasure}]

The inequality $C^{d^k}_{prod}(\mathcal E_\lambda^{\otimes_k})\geq  kC^d_{prod}(\mathcal E_\lambda)$ holds for every channel (since one could use each copy of the channel independently). According to Theorem \ref{Main theorem Depolarizing} this implies that $C^{d^k}_{prod}(\mathcal E_\lambda^{\otimes_k})\geq k\lambda \ln  (nd)$. On the other hand, according to Lemmas \ref{lemma:Upper bound C^d_{prod} erasure} and  Lemma \ref{lemma add erasure} we have
\begin{align*}
C^{d^k}_{prod}(\mathcal E_\lambda^{\otimes_k})&\leq \sum_{s=1}^k\binom{k}{s}\lambda^s(1-\lambda)^{k-s}\ln  n^s+\sum_{s=0}^k\binom{k}{s}\lambda^s(1-\lambda)^{k-s}V_{d^k}(\Ne_s)\\&\leq
\sum_{s=1}^k\binom{k}{s}\lambda^s(1-\lambda)^{k-s}\ln  n^s+\sum_{s=0}^k\binom{k}{s}\lambda^s(1-\lambda)^{k-s}\frac{s}{k}\ln  d^k\\&=\sum_{s=1}^k\binom{k}{s}\lambda^s(1-\lambda)^{k-s}s(\log  (nd))\\&=k\lambda\ln  (nd).
\end{align*}Here, we have used that
\begin{align*}
\sum_{s=1}^k\binom{k}{s}\lambda^s(1-\lambda)^{k-s}s=k\lambda.
\end{align*}In order to see this, let us proceed by induction.

For $k=2$ we have $\sum_{s=1}^2\binom{k}{s}\lambda^s(1-\lambda)^{k-s}s=\binom{2}{1}\lambda(1-\lambda)+\binom{2}{2}\lambda^22=2\lambda$. Let us now assume the result for $k$. Then,
\begin{align*}
\sum_{s=1}^{k+1}\binom{k+1}{s}\lambda^s(1-\lambda)^{k+1-s}s&=\lambda(k+1)\sum_{s=1}^{k+1}\binom{k}{s-1}\lambda^{s-1}(1-\lambda)^{k-(s-1)}\\&=
\lambda(k+1)\sum_{s=0}^{k}\binom{k}{s}\lambda^{s}(1-\lambda)^{k-s}=\lambda(k+1),
\end{align*}where in the last equality we have used again the identity (\ref{identity}).

This finishes the proof.
\end{proof}
Lemma \ref{lemma add erasure} can be obtained as a simple consequence of the following deep and extremely useful result in information theory.
\begin{theorem}[Strong subadditivity inequality, \cite{LiRu}]\label{SSI}
For every tripartite state $\rho\in S_1\otimes S_1^n\otimes S_1^n$ the following inequality holds.
\begin{align*}
S(\rho)+S\Big(\big(tr_n\otimes id_n \otimes tr_n\big)(\rho)\Big)\leq S\Big(\big(id_n\otimes id_n \otimes tr_n\big)(\rho)\Big)+S\Big(\big(tr_n\otimes id_n\otimes id_n\big)(\rho)\Big).
\end{align*}
\end{theorem}
In general, if we call the respective systems $A$, $B$ and $C$, the strong subadditivity inequality can be written by $$S(ABC)+S(B)\leq S(AB)+S(BC).$$Of course, the system $B$ can be replaced by system $A$ and $C$ and the analogous inequality holds. It is also interesting to mention that the stong subadditivity inequality can be obtained by differentiating the norm $\big\|\rho \big\|_{S_1[S_p]}$ and using a Minkowski-type inequalities (see \cite[Section 6]{DJKR}).

We thank Andreas Winter for the explanation of the following proof which simplified very much a previous proof by the authors (not using the strong subadditivity inequality).
\begin{proof}[Proof of Lemma \ref{lemma add erasure}]
According to our definition (\ref{V_d}), $V_d(\Ne_s)$ can be trivially written as
\begin{align*}
\sup \Big\{\frac{s}{k}S\big((id_d\otimes tr_{n^k})(\rho)\big) +\Big(\frac{k-s}{k}\Big)S\big((id_d\otimes tr_{n^k})(\rho)\big)-  \frac{1}{\binom{k}{s}}\sum_{\substack{A\subseteq \{1,\cdots, k\}\\|A|=s}}S\big((id_d\otimes \Ne_A)(\rho)\big) \Big\},
\end{align*}where here the supremum is taken over all pure states $\rho\in S_1^d(S_1^{n^k})$. Since, we clearly have $S\big((id_d\otimes tr_{n^k})(\rho)\big) \leq \ln d$ for every state $\rho$, it suffices to show that for every pure state $\rho\in S_1^d(S_1^{n^k})$ we have
\begin{align*}
\Big(\frac{k-s}{k}\Big)S\big((id_d\otimes tr_{n^k})(\rho)\big)\leq  \frac{1}{\binom{k}{s}}\sum_{\substack{A\subseteq \{1,\cdots, k\}\\|A|=s}}S\big((id_d\otimes \Ne_A)(\rho)\big).
\end{align*}Now, since we are assuming that $\rho$ is pure, the previous inequality is the same as
\begin{align*}
\Big(\frac{k-s}{k}\Big)S\big((tr_d\otimes id_{n^k})(\rho)\big)\leq  \frac{1}{\binom{k}{s}}\sum_{\substack{A\subseteq \{1,\cdots, k\}\\|A|=s}}S\big((tr_d\otimes \Ne_{A^c})(\rho)\big).
\end{align*}Since we must prove the result for every  $0\leq s\leq k$, by replacing $s$ with $k-s$, we see that it suffices to show that for every not necessarily pure state $\rho\in S_1^{n^k}$ and for every $0\leq s\leq k$ one has
\begin{align*}
\frac{s}{k}S(\rho)\leq  \frac{1}{\binom{k}{s}}\sum_{\substack{A\subseteq \{1,\cdots, k\}\\|A|=s}}S\big(\Ne_{A}(\rho)\big).
\end{align*}Let us simplify the notation of the previous inequality by writing it as
\begin{align}\label{s-general erasure add}
\frac{s}{k}S(A_1\cdots A_k)\leq  \frac{1}{\binom{k}{s}}\sum_{|\delta|=s}S\big(A_\delta \big),
\end{align}with the obvious interpretation.
We will first prove this inequality for the particular case $s=k-1$ and we will obtain the general case by induction. In this case, we must show
\begin{align}\label{s=k-1erasure add}
(k-1)S(A_1\cdots A_k)\leq  \sum_{i=1}^kS\big(A_{[k]-\{i\}}\big),
\end{align}Let us consider a purification\footnote{Given any state $\rho\in S_1^N$, we can always find a unit vector $|\psi\rangle\in \C_M\otimes \C_N$ so that $(tr_M\otimes id_N)\big(|\psi\rangle\langle\psi|\big)=\rho$.} $WA_1\cdots A_k$ of the system $A_1\cdots A_k$ (that is, the state $\rho\in S_1^{n^k}$) so that we can write the previous expression as
\begin{align}\label{s=k-1erasure add alternative}
(k-1)S(W)\leq  \sum_{i=1}^kS\big(WA_i\big).
\end{align}Here, we are using that for every multipartite pure state the von Neumann entropy of any subsystem is the same as the von Neumann entropy of the complement subsystem, which is a direct consequence of the Hilbert Schmidt decomposition. Now, a direct application of Theorem \ref{SSI} implies that for every $0\leq s\leq k-1$, $$S(W)+S(WA_{[k]-\{0,\cdots ,s\}})\leq S(WA_{s+1})+S(WA_{[k]-\{0,\cdots ,s+1\}}).$$ Then, we can obtain Equation (\ref{s=k-1erasure add alternative}) by applying this inequality k-1 times iterately.
%
%
%
\begin{comment}
Indeed, we will have
%
\begin{align*}
&(k-1)S(W)- \sum_{i=1}^kS\big(WA_i\big)\leq (k-1)S(W)+S(WA_1\cdots A_k)- \sum_{i=1}^kS\big(WA_i\big)\\&=
(k-2)S(W) +S(W)+S(WA_1\cdots A_k)- S(WA_1)-\sum_{i=2}^kS\big(WA_i\big)\\&\leq^{s=0} (k-2)S(W)+S(WA_2\cdots A_k)- \sum_{i=2}^kS\big(WA_i\big)\\&=
(k-3)S(W)+S(W)+S(WA_2\cdots A_k)- S(WA_2)-\sum_{i=3}^kS\big(WA_i\big)\\&\leq^{s=1}(k-3)S(W)+S(WA_3\cdots A_k)-\sum_{i=3}^kS\big(WA_i\big)\\&\cdots
\end{align*}after applying k-1 steps we have $S(WA_k)-S(WA_k)=0$.
\end{comment}
%
%
%
With Equation (\ref{s=k-1erasure add}) at hand, we can finish our proof by using induction. Checking that (\ref{s-general erasure add}) holds for $k=2$ ($s=0,1,2$) is very easy by just using the subadditivity of the von Neumann entropy\footnote{This result is a trivial consequence of Theorem \ref{SSI}.} $S(A_1A_2)\leq S(A_1)+S(A_2)$. On the other hand, let us assume that (\ref{s-general erasure add}) holds for every state $\rho\in S_1^{n^{k-1}}$ (so for every  systems $A_1,\cdots ,A_{k-1}$) and every $0\leq s\leq k-1$ and we will show that, then, it must also hold for $k$. First of all, note that the case $s=k$ is completely trivial, so it suffices to consider $0\leq s\leq k-1$. Then, we can write
\begin{align*}
\frac{s}{k}S(A_1\cdots A_k)&\leq  \frac{s}{k(k-1)} \sum_{i=1}^kS\big(A_{[k]-\{i\}}\big)\leq  \frac{s}{k(k-1)}\sum_{i=1}^k\frac{k-1}{s}\frac{1}{\binom{k-1}{s}}\sum_{|\delta|=s}S\big(A_{[k]-\{i\};\delta} \big)\\&=\frac{1}{k} \sum_{i=1}^k\frac{1}{\binom{k-1}{s}}\sum_{|\delta|=s}S\big(A_{[k]-\{i\};\delta}  \big)=\frac{1}{\binom{k}{s}}\sum_{|\delta|=s}S\big(A_\delta \big).
\end{align*}Here, the first inequality follows from Equation (\ref{s=k-1erasure add}) and the second inequality follows from the induction hypothesis. The last equality is straighforward.
%Just note that $\sum_{|\delta|=s}S\big(A_\delta \big)=\frac{1}{k-s}\sum_{i=1}^k\sum_{|\delta|=s}S\big(A_{[k]-\{i\};\delta}\big)$.
\end{proof}
\section*{Acknowledgments}
We thank Andreas Winter and Toby S. Cubitt for helpful conversations. A part of this work was done at the Isaac Newton Institute (Cambridge, U.K.), during the programme on Mathematical Challenges in Quantum Information in Fall 2013.

\vskip 2cm

\hfill \noindent \textbf{Marius Junge} \\
\null \hfill Department of Mathematics \\ \null \hfill
University of Illinois at Urbana-Champaign\\ \null \hfill 1409 W. Green St. Urbana, IL 61891. USA
\\ \null \hfill\texttt{junge@math.uiuc.edu}

\vskip 1cm

\hfill \noindent \textbf{Carlos Palazuelos} \\
\null \hfill Instituto de Ciencias Matem\'aticas, ICMAT\\
\null \hfill Facultad de Ciencias Matem\'aticas\\ \null \hfill
Universidad Complutense de Madrid \\ \null \hfill Plaza de Ciencias s/n.
28040, Madrid. Spain
\\ \null \hfill\texttt{carlospalazuelos@ucm.es}

%(n=4, d=2, k=2) (x+(1-x)/16)log_2(x+(1-x)/16)+(15/16)(1-x)log_2((1-x)/16)+(x+(1-x)/4)log_2(x+(1-x)/4)+(3/4)(1-x)log_2((1-x)/4)-2*(x+(1-x)/8)log_2(x+(1-x)/8)-(14/8)(1-x)log_2((1-x)/8) in (0,1)

% (n=3, d=3,1) log_2(27/36)+(x+(1-x)/9)log_2(x+(1-x)/9)+(8/9)(1-x)log_2((1-x)/9)+(x+(1-x)/3)log_2(x+(1-x)/3)+(2/3)(1-x)log_2((1-x)/3)-2*(x+(1-x)/6)log_2(x+(1-x)/6)-(10/6)(1-x)log_2((1-x)/6) in (0,1)

\end{document}